\documentclass[a4paper]{article} % for Latex2e

\usepackage[usenames,dvipsnames,table]{xcolor}
\definecolor{lightgray}{gray}{0.85}

\usepackage{enumitem}

\usepackage{graphicx} % more modern
\usepackage{subfigure}

\usepackage[round,sort,numbers, authoryear]{natbib}
\usepackage[lined,longend,algoruled]{algorithm2e} %,linesnumbered
\DontPrintSemicolon
\usepackage{algorithmic}

\usepackage{hyperref,a4wide}

\usepackage{pdflscape}

\usepackage{amsmath,amssymb,amsfonts,amsthm}% personnal addition
\usepackage{dsfont}
\usepackage{stmaryrd} % for \llbracket
\usepackage{array} % for \newcolumntype
\usepackage{nonfloat,mathrsfs,booktabs}

\usepackage{multirow}

\def\diag{{\rm diag}}
\def\rank{{\rm rank}}
\def\1{\mathds 1}

\def\hat{\widehat}
\def\tilde{\widetilde}

\def\prob{{\mathbf P}}

\def\trace{\text{ \rm trace}}

\def\bfone{{\mathbf 1}}

\def\bfc{{\mathbf c}}
\def\bfe{{\mathbf e}}

\def\bfB{{\mathbf B}}
\def\bfM{{\mathbf M}}

\def\bfU{{\mathbf U}}
\def\bfX{{\mathbf X}}
\def\bfY{{\mathbf Y}}
\def\bfI{{\mathbf I}}
\def\calJ{{\mathcal J}}
\def\calO{{\mathcal O}}

\def\bfA{{\mathbf A}}

\def\bfD{{\mathbf D}}
\def\bfE{{\mathbf E}}

\def\bfP{{\mathbf P}}
\def\bfZ{{\mathbf Z}}
\def\bfSigma{\mathbf\Sigma}
\def\bmu{\boldsymbol\mu}

\def\btheta{\boldsymbol\theta}

\def\bepsilon{\boldsymbol\epsilon}
\def\bxi{\boldsymbol\xi}
\def\bphi{\boldsymbol\phi}

\def\bfu{{\mathbf u}}
\def\bfv{{\mathbf v}}
\def\bfw{{\mathbf w}}

\def\bfD{{\mathbf D}}
\def\bfP{{\mathbf P}}
\def\bfV{{\mathbf V}}

\def\bfDelta{\mathbf\Delta}
\def\bfTheta{\mathbf\Theta}
\def\bfPi{\mathbf\Pi}
\def\bfOmega{\mathbf\Omega}

\newcommand{\field}[1]{\mathbb{#1}}

\newcommand{\R}{\field{R}}

\DeclareMathOperator{\e}{e}

\DeclareMathOperator*{\argmin}{arg\,min}

\newtheorem{theorem}{Theorem}    % numérotés par section
\newtheorem{prop}{Proposition}    % Les propositions ont le même compteur
\newtheorem{lemma}{Lemma}    % Les lemmes ont le même compteur
\newtheorem{remark}{Remark}    % Les lemmes ont le même compteur
\begin{document}

\title{Convex programming approach to robust estimation of a multivariate Gaussian model}
%\date{27 avril 2015}
\author{Samuel Balmand and Arnak Dalalyan}

\maketitle

\begin{abstract}
Multivariate Gaussian is often used as a first approximation to the distribution
of high-dimensional data. Determining the parameters of this distribution under various constraints
is a widely studied problem in statistics, and is often considered as a prototype for testing
new algorithms or theoretical frameworks. In this paper, we develop a nonasymptotic approach to
the problem of estimating the parameters of a multivariate Gaussian distribution when data are
corrupted by outliers.  We propose an estimator---efficiently computable by solving a convex
program---that robustly estimates the population mean and the population covariance matrix even
when the sample contains a significant proportion of outliers. Our estimator of the corruption
matrix is provably rate optimal simultaneously for the entry-wise $\ell_1$-norm, the Frobenius
norm and the mixed $\ell_2/\ell_1$ norm. Furthermore, this optimality is achieved by a penalized
square-root-of-least-squares method with a universal tuning parameter (calibrating the strength
of the penalization). These results are partly extended to the case where $p$ is potentially larger
than $n$, under the additional condition that the inverse covariance matrix is sparse.
\end{abstract}

%\linenumbers

\section{Introduction}

In many applications where statistical methodology is employed, multivariate Gaussian
distribution plays a central role as a first approximation to the distribution of high-dimensional
data. It is mainly motivated by the fact that high dimensional data, being sparsely distributed in
space, can be reasonably well fitted by an elliptically countered distribution, of which the Gaussian
distribution is the most famous representative. Another reason is that in high-dimensional inference,
sophisticated nonparametric methods suffer from the curse of dimensionality and lead to poor results
(both in theory and in practice). For these reasons, recent years have witnessed an increased interest
for simple parametric models in the statistical literature, with a particular emphasis on the effects
of high-dimensionality and the relevance of developing nonasymptotic theoretical guarantees. In this
context, Gaussian models play a particular role in relation with the graphical modeling and discriminant
analysis, but also because they provide a convenient theoretical framework for showcasing new ideas
and analyzing new algorithms.

Determining the parameters of the Gaussian distribution under various constraints is a widely studied
problem in statistics. Recent developments around sparse coding and compressed sensing have opened new
lines of research on Gaussian models in which classical estimators such as the ordinary least squares
and the empirical covariance matrix are strongly sub-optimal. Novel statistical procedures---often
based on convex optimization---have emerged to cope with the aforementioned sub-optimality of traditional
techniques. In addition, establishing nonasymptotic theoretical guarantees that highlight the impact of
the dimensionality and the level of sparsity has appeared as a primary target of theoretical studies.
The present work continues this line of research by developing a nonasymptotic approach to the
problem of estimating the parameters of a multivariate Gaussian distribution from a sample of independent
and identically distributed observations corrupted by outliers.

We propose an estimator---efficiently computable by solving a convex program---that robustly estimates the
population mean and the population (inverse) covariance matrix even when the sample contains a significant
proportion of outliers. The estimator is defined as the minimizer of a cost function that combines a data
fidelity term with a sparsity-promoting penalization. Following and extending the methodology developed in
\citep{BCW,SunZhang12}, the data fidelity term is defined as the mixed $\ell_2/\ell_1$ norm of the residual
matrix. The penalty term is proportional to the mixed $\ell_2/\ell_1$ norm of a matrix that models
the outliers. Our estimator of the corruption matrix is proved to be rate optimal simultaneously for the
entry-wise $\ell_1$-norm, the Frobenius norm and the mixed $\ell_2/\ell_1$ norm. Furthermore, this
optimality is achieved by a penalized square-root of least squares method with a universal tuning parameter
calibrating the magnitude of the penalty.

The results are partly extended to the case where $p$ is potentially larger than $n$, but the inverse covariance
matrix is sparse. In such a situation, we recommend to add to the cost function an additional penalty term that
corresponds, to some extent, to a weighted entry-wise $\ell_1$ norm of the inverse covariance matrix. The theoretical
guarantees established in this case are not as complete and satisfactory as those of low/moderate dimensional case.
In particular, the obtained risk bounds are valid in the event that the empirical covariance matrix satisfies
a particular type of restricted eigenvalues condition \citep{BRT}. At this stage, we are not able to theoretically
assess the probability of this event. Another open problem is the practical choice of the tuning parameter.
We are currently working on these issues and hope to address them in a forthcoming paper.

\subsection{Mathematical framework}

We adopt here the following formalization of the multivariate Gaussian model in presence of outliers.
We assume that the outlier-free data $\bfY$ consists of $n$ row-vectors independently drawn from a multivariate Gaussian
distribution with mean $\bmu^*$ and covariance matrix $\bfSigma^*$, hereafter denoted by $\mathcal N_p(\bmu^*,\bfSigma^*)$.
However, the data $\bfY$ is revealed to the Statistician after being corrupted by outliers. So, the Statistician has access
to a data matrix $\bfX\in\R^{n\times p}$ satisfying
\begin{align}\label{outlier-model}
\bfX = \bfY+\bfE^*.
\end{align}
The matrix of errors $\bfE^*$ has a special structure: most rows of $\bfE^*$---corresponding to inliers---have only zero entries.
We will denote by $O$ the subset of indices from $\{1,\ldots,n\}$ corresponding to the outliers and by $I=\{1,\ldots,n\}\setminus O$
the subset of inliers. The following two conditions will be assumed throughout the paper:
\begin{description}
\item[\bf (C1)] The $n$ rows of the matrix $\bfY$ are independent $\mathcal N_p(\bmu^*,\bfSigma^*)$ random vectors.
\item[\bf (C2)] The $n\times p$ contamination matrix $\bfE^*$ is deterministic and, for every $i\in I\subset\{1,\ldots,n\}$, the $i$-th row of
$\bfE^*$ is zero. Furthermore, the rows of $\bfE^*(\bfSigma^*)^{-1/2}$ are bounded in Euclidean norm by $M_\bfE\sqrt{p}$, for some
constant $M_\bfE$.
\end{description}
For an introduction to the problem of robust estimation in statistics, we refer the reader to
\citep{Huber09,Hampel,MaronnaMartinYohai2006}. An overview of  more recent advances closely related to
the present work can be found in \citep{Gao2015,LohTan2015}.

\subsection{Notation}

%The expectation of a random vector $\bfx$ is denoted by $\Ex(\bfx)$ and its covariance matrix by $\Var(\bfx)$.
%We denote by $\Phi$ and $\varphi$ the distribution function and the density function of the standard normal law, thus $\Phi^{-1}$ is the standard normal quantile function. We also use the notation $\bar\Phi = 1 - \Phi$.
We denote by $\bfone_n$ the vector from $\R^n$ with all the entries equal to $1$ and by
$\bfI_n$ the $n\times n$ identity matrix.  We write $\mathds{1}$ for the indicator function,
which is equal to 1 if the considered condition is satisfied and 0 otherwise.
%The smallest integer larger than or equal to $x \in \R$ is denoted by $\lceil x \rceil$.
The cardinality of a set $S$ is denoted by $|S|$. In what follows, $[p] :=\{1,\ldots,p\}$
is the set of integers from $1$ to $p$. For $i \in [p]$, the complement of the singleton
$\{i\}$ in $[p]$ is denoted by $i^c$. For a vector $\bfv\in\R^p$, $\bfD_\bfv$ stands for
the $p\times p$ diagonal matrix satisfying $(\bfD_\bfv)_j = \bfv_j$
for every $j\in[p]$. The matrix obtained from $\bfM$ by zeroing its off-diagonal entries
is denoted by $\diag(\bfM)$.

The transpose of the matrix $\bfM$ is denoted by $\bfM^\top$. %If this matrix is square, we note $|\bfM|$ its determinant.
The sub-vector of a vector $\bfv \in \R^p$ obtained by removing all the elements with indices
in $J^c\subset[p]$ is denoted by $\bfv_J$. For a $n\times p$ matrix $\bfM$, we denote by denoted by
$\bfM_{k,J}$ (resp.\ $\bfM_{K,j}$) the vector formed by the entries of the $k$-th row (resp.\ the $j$-th column)
of $\bfM$ whose indices are in the subset $J$ of $[p]$ (resp.\ $K$ of $[n]$). In particular, $\bfM_{k^c,j}$ stands for
the vector made of all the entries of the $j$-th column of the matrix $\bfM$ at the exception of the element of the
$k$-th row. Moreover, the whole $k$-th row (resp.\ $j$-th column) of $\bfM$ is denoted by $\bfM_{k,\bullet}$
(resp.\ $\bfM_{\bullet,j}$). We use the following notation for the (pseudo-)norms of matrices: if $q_1,q_2>0$, then
\begin{align*}
{\|\bfM\|}_{q_1,q_2} =
\bigg\{\sum_{i=1}^n \|\bfM_{i,\bullet}\|_{q_1}^{q_2}\bigg\}^{1/q_2} .
\end{align*}
With this notation, $\|\bfM\|_{2,2}$ and $\|\bfM\|_{1,1}$ are the Frobenius, also denoted by ${\|\bfM\|}_F$,
and the element-wise $\ell_1$-norm of $\bfM$, respectively. One or both of the parameters $q_1$ and $q_2$
may be equal to infinity. In particular, the element-wise $\ell_\infty$-norm of $\bfM$ is defined by
${\|\bfM\|}_{\infty,\infty} = \max_{(i,j)\in[n]\times[p]}|\bfM_{i,j}| = \max_{j\in[p]}\|\bfM_{\bullet,j}\|_\infty$
and we denote ${\|\bfM\|}_{2,\infty} = \max_{i\in[n]}{\|\bfM_{i,\bullet}\|}_2$. We also define
$\sigma_{\max}(\bfM)$ and $\sigma_{\min}(\bfM)$, respectively, as the largest and the smallest singular values of the
matrix $\bfM$. Finally, $\bfM^\dag$ stands for the Moore-Penrose pseudo-inverse of a matrix $\bfM$.

\subsection{Robust estimator by convex programming}

In the situation under investigation in this work, it is assumed that the sample contains some outliers. In other terms, the
relation $\bfX_{i,\bullet}\sim \mathcal N_p(\bmu^*,\bfSigma^*)$ holds true only for indices $i$ belonging to some subset
$I$ of $[n]$. The set $I$ is large, but does not necessarily coincide with the entire set $[n]$. In such  a context, our 
proposal consist in extending the methodology developed in \citep{SunZhang13}. Recall that in the case when no outlier is 
present in the sample, the scaled lasso \citep{SunZhang13} estimates the matrix $\bfOmega^*=(\bfSigma^*)^{-1}$
by first solving the optimization problem
\begin{align}\label{step:1}
\hat\bfB = \argmin_{\substack{\bfB:\bfB_{jj}=1}}\limits\ \min_{\bfc\in\R^p}
\Big\{{\|(\bfX\bfB- \bfone_n \bfc^\top)^\top\|}_{2,1}+\bar\lambda{\|\bfB\|}_{1,1}\Big\},
\end{align}
for a given tuning parameter $\bar\lambda\ge0$, where the $\argmin$ is over all $p\times p$ matrices $\bfB$ having all their 
diagonal entries equal to $1$. The second step of the scaled lasso is to set
\begin{align}\label{step:2}
\hat\omega_{jj} = \Big(\frac1n{\|(\bfI_n-n^{-1}\bfone_n\bfone^\top_n)\bfX\hat\bfB_{\bullet,j}\|}_{2}^2\Big)^{-1};
\qquad \hat\bfOmega = \hat\bfB\cdot\text{diag}(\{\hat\omega_{jj}\}_{j\in[p]}).
\end{align}
%Considering problem \eqref{step:1}, we observe that each column of $\bfB^*$ can be estimated separately by solving
%\begin{align}\label{eq:optProbCol}
%\hat\bfB_{\bullet,j} = \argmin_{\substack{\bbeta\in \R^{p}\\ \bbeta_{j}=1}}\limits\min_{\bfc\in\R^p}
%\Big\{{\|\bfX\bbeta - \bfc_j\bfone_n\|}_2+\bar\lambda{\|\bbeta\|}_1\Big\} .
%\end{align}
%This last estimator is nothing but the square-root lasso estimator based on the regression of each column of $\bfX$ using every others as explanatory variables.

In the case of observations corrupted by outliers, we propose to modify the scaled lasso procedure as follows.
Let us denote by $\bfu_n$ the vector $\bfone_n/\sqrt{n}$ and by $\bfX^{(n)}$ the matrix $\bfX/\sqrt{n}$. This scaling is convenient
since it makes the columns of the data matrix to be of a nearly constant Euclidean norm, at least in the case without outliers.
We replace step (\ref{step:1}) by
\begin{align}\label{step:1bis}
\{\hat\bfB,\hat\bfTheta\} = \argmin_{\substack{\bfB :\bfB_{jj}=1\\ \bfTheta\in\R^{n\times p}}}\limits\ \min_{\bfc\in\R^p}
\bigg\{ {\big\|(\bfX^{(n)}\bfB-\bfu_n\bfc^\top-\bfTheta)^\top\big\|}_{2,1}+\lambda \big({\|\bfTheta\|}_{2,1} + \gamma {\|\bfB\|}_{1,1}\big)\bigg\},
\end{align}
where $\lambda\ge 0$ is a tuning parameter associated with the regularization term promoting robustness
and where $\lambda\gamma \ge 0$ corresponds to the tuning parameter whose aim is to encourage sparsity of the matrix $\bfB$
(or, equivalently, of the corresponding graph). Using the estimators $\{\hat\bfB,\hat\bfTheta\}$,
the entries of the precision matrix $\bfOmega^*$ are estimated by
\begin{align}\label{step:2bis}
\hat\omega_{jj} = \frac{2n}{\pi} {\|(\bfI_n-\bfu_n\bfu^\top_n)(\bfX^{(n)}\hat\bfB_{\bullet,j}-\hat\bfTheta_{\bullet,j})\|}_{1}^{-2};
\qquad \hat\bfOmega = \hat\bfB\cdot\text{diag}(\{\hat\omega_{jj}\}_{j\in[p]}).
\end{align}
The matrix $\bfE^*$ and the vector $\bmu^*$ can be estimated by
\begin{align}\label{step:3bis}
\hat\bfE = \sqrt{n}\, \hat\bfTheta\hat\bfB^\dag\qquad\text{and}\qquad\hat\bmu = \frac1n (\bfX-\hat\bfE)^\top\bfone_n.
\end{align}
It is important to stress right away that the robust estimation procedure described by equations \eqref{step:1bis}-\eqref{step:3bis}
can be efficiently realized in practice even for large dimensions $p$. Indeed, the first step boils down to solving a convex program, that
can be cast into a second-order cone program, whereas the two last steps involve only simple operations with matrices and vectors.

To explain the rationale behind this estimator, let us recall the following well-known result concerning multivariate Gaussian distribution.
If we denote $\bfB^* = \bfOmega^*\text{diag}(\bfOmega^*)^{-1}$, then we have
\begin{align*}
\big(\bfY-\bfone_n(\bmu^*)^\top\big)\bfB^*_{\bullet,j} = \phi^*_{j} \,\bepsilon_{\bullet,j},
\end{align*}
where $\bepsilon_{\bullet,j}\sim \mathcal N_n(0,\bfI_n)$ is a random vector independent of $\bfY_{\bullet,j^c}$ and $\phi^*_j = (\omega^*_{jj})^{-1/2}$.
Combining this relation with \eqref{outlier-model} and using the notations $\bfTheta^*=\bfE^*\bfB^*/\sqrt{n}\in\R^{n\times p}$
and $\bfc^* = (\bfB^*)^\top\bmu^*$, we get
\begin{align}\label{main:eq}
\bfX^{(n)}\bfB^*_{\bullet,j} = c_j^*\bfu_n+\bfTheta^*_{\bullet,j}+ \frac{\phi^*_{j}}{\sqrt{n}} \,\bepsilon_{\bullet,j},\qquad\forall j\in[p].
\end{align}
Furthermore, the matrix $\bfTheta^*$ inherits the row-sparse structure of the matrix $\bfE^*$ whereas the matrix $\bfB^*$ has exactly
the same sparsity pattern as the precision matrix $\bfOmega^*$. This suggests to recover the triplet $(\bfc^*,\bfB^*,\bfTheta^*)$ by
minimizing a penalized loss where the penalty imposed on $\bfTheta$ promotes the row-sparsity, while the penalty imposed on $\bfB$
favors sparse matrices without any particular structure of the sparsity pattern. It is well known in the literature on group sparsity
(see \cite{lounici2011} and the references therein) that the mixed $\ell_2/\ell_1$-norm penalty $\|\,\!\cdot\,\!\|_{2,1}$ is well suited
for taking advantage of the row-sparsity while preserving the convexity of the penalty. A more standard application of the lasso
to our setting would suggest to use the residual sum of squares ${\big\|\bfX^{(n)}\bfB-\bfu_n\bfc^\top-\bfTheta\big\|}^2_{2,2}$
as the data fidelity term, instead of the mixed $\ell_2/\ell_1$-norm written in \eqref{step:1bis}. However, similarly to the square-root
lasso~\citep{BCW}, and as shown in the results of the next sections, the latter has the advantage of making the tuning parameter
$\lambda$ scale free. It allows us to define a universal value of $\lambda$ that does not depend on the noise levels $\phi_j^*$ in
Eq.\ \eqref{main:eq} and, nevertheless, leads to rate optimal risk bounds. 

Note that during the past ten years several authors proposed to employ convex penalty based approaches to
robust estimation in various settings, see for instance \citep{Candes08,DalalyanK,Nguyen,DalalyanC12}.
The problems considered in these papers concern the estimation of a vector parameter and do not directly
carry over the problem under investigation in the present work.

From the theoretical point of view, analyzing statistical properties of the estimators $\hat\bfTheta$, $\hat\bfB$ and
$\hat\bfOmega$ turns out to be a challenging task. Indeed, despite the obvious similarity of problem (\ref{step:1bis}) to its
vector regression counterpart \citep{BCW,SunZhang12}, optimization problem \eqref{step:1bis} contains an important
difference: the objective function is not decomposable with respect to neither rows nor columns of the matrix $\bfTheta$.
In fact, the objective is the sum  of two terms, the first being decomposable with respect to the columns of $\bfTheta$
and non-decomposable with respect to the rows, while the second is decomposable with respect to the rows but non-decomposable
with respect to the columns. As shown in the theorems stated below as well as in their proofs, we succeeded in overcoming
this difficulty by means of nontrivial combinations of elementary arguments. We believe that some of the tricks used in the
proofs may be useful in other problems where the objective function happens to be non-decomposable.

The rest of the manuscript is organized as follows. Having already introduced the proposed method for robust estimation of a 
sparse precision matrix, we present our main theoretical findings in Section~\ref{sec:3}. A discussion on the advantages and
limitations of the obtained results as compared to previous work on robust estimation, as well as extensions
to high dimensional setting, are included in Section~\ref{sec:4}. Technical proofs are postponed to Section~\ref{sec:6}, 
whereas some promising numerical results are reported in Section~\ref{sec:5}.

\section{Moderate dimensional case: theoretical results}\label{sec:3}

In order to ease notation and to avoid some technicalities that may blur the main ideas, we assume that $\bmu^*=0$ which 
implies that $\bfc^* = 0$, see~Eq.~\eqref{main:eq}, and we do not need to minimize with respect to $\bfc$ in (\ref{step:1bis}).
We introduce the (unnormalized) residuals  $\bxi_{\bullet,j} =\bphi^*_{j}\,\bepsilon_{\bullet,j}/\sqrt{n}$, so that the following 
relation holds:
\begin{align}\label{eq:xi}
 \bfX^{(n)}\bfB^* = \bfTheta^*+ \bxi .
\end{align}
For a better understanding of the assumptions that are needed to establish a tight upper bound on the error of
estimation of the matrix $\bfB^*$ of coefficients and the matrix $\bfTheta^*$ corresponding to the outliers, we start by
analyzing the problem of robust estimation when $p$ is of smaller order than $n$, and no sparsity assumption on $\bfOmega^*$
is made. We call this setting the moderate dimensional case, since we allow the dimension to go to infinity with the sample
size, provided that the ratio $p/n$ remains small\footnote{This is different from the ``low dimensional case'' in which
$p$ is assumed fixed when $n$ goes to infinity, so that the quantities depending only on $p$ are treated as constants.}.
In such a situation there is no longer need to penalize nonsparse matrices $\bfB$ in the
optimization problem. We work with the estimator
\begin{align}\label{eq:1L}
\{\hat\bfB,\hat\bfTheta\} = \argmin_{\substack{\bfB\in \R^{p\times p}\\ \bfB_{jj}=1}}\limits\min_{\bfTheta\in\R^{n\times p}}
\bigg\{ {\big\|(\bfX^{(n)}\bfB-\bfTheta)^\top\big\|}_{2,1}+\lambda{\|\bfTheta\|}_{2,1}\bigg\}.
\end{align}
For a given matrix $\bfTheta$,  the minimum with respect to $\bfB$ in the foregoing optimization problem is a solution to
the convex program
\begin{align}\label{eq:2L}
\hat\bfB(\bfTheta) = \argmin_{\substack{\bfB\in \R^{p\times p}\\ \bfB_{jj}=1}}
\bigg\{ \sum_{j=1}^p {\big\|\bfX^{(n)}_{\bullet,j}-\bfTheta_{\bullet,j} + \bfX^{(n)}_{\bullet,j^c}\bfB_{j^c,j}\big\|}_2\bigg\},
\end{align}
which decomposes into $p$ independent ordinary least squares problems.
A solution of the latter is provided by the formula
\begin{align}
 \bfX^{(n)}_{\bullet,j^c}\hat\bfB_{j^c,j}(\bfTheta) = - \bfPi^{j^c}(\bfX^{(n)}_{\bullet,j} - \bfTheta_{\bullet,j}) \quad \text{and} \quad \hat \bfB_{jj}(\bfTheta) = 1, \label{eq:4L}
\end{align}
where the notation $\bfPi^{j^c}$ is used for the orthogonal projector in $\R^n$ onto the subspace spanned by the columns of $\bfX^{(n)}_{\bullet,j^c}$. Let us introduce now the matrices $\bfZ^j = \bfI_n - \bfPi^{j^c}$ that are orthogonal projectors onto the orthogonal complement of the linear subspace of $\R^n$ spanned by the columns of $\bfX_{\bullet,j^c}$ (or, equivalently, of $\bfX^{(n)}_{\bullet,j^c}$). Using this notation and replacing expression \eqref{eq:4L} in problem \eqref{eq:1L}, we arrive at
\begin{align}\label{eq:5L}
\hat\bfTheta = \argmin_{\bfTheta\in\R^{n\times p}}
\bigg\{ \sum_{j=1}^p {\big\|\bfZ^j (\bfX^{(n)}_{\bullet,j} - \bfTheta_{\bullet,j}) \big\|}_2 +\lambda{\|\bfTheta\|}_{2,1}\bigg\} .
\end{align}
In what follows, we rely on formulae \eqref{eq:5L} and \eqref{eq:4L} both for computing and analyzing the estimator provided by Eq.~\eqref{eq:1L}.
Our first result concerns the quality of estimating the outlier matrix $\bfTheta^*$.

\begin{theorem}\label{thm:1}
Let assumptions (C1) and (C2) be satisfied.
Let $\delta\in(0,1)$ such that $n\ge |O|+8p+16\log(4/\delta)$ and choose
\begin{equation}\label{lambd}
\lambda = 6 \bigg(\frac{p\log(2np/\delta)}{n}\bigg)^{1/2}.
\end{equation}
If $40|O|p(13\log(2np/\delta)+2(1+M_{\bfE})^2)\le n-|O|$, then with probability at least $1-3\delta$,
\begin{align}
{\|\hat\bfTheta-\bfTheta^*\|}_{1,1}&\le 3C_0 \max_{j} (\omega_{jj}^*)^{-1/2} |O|p\bigg(\frac{\log(2np/\delta)}{n}\bigg)^{1/2},\label{rhs1}\\
{\|\hat\bfTheta-\bfTheta^*\|}_{2,1}&\le 3C_0 \max_{j} (\omega_{jj}^*)^{-1/2} |O|\bigg(\frac{p\log(2np/\delta)}{n}\bigg)^{1/2},\label{rhs2}\\
{\|\hat\bfTheta-\bfTheta^*\|}_{2,2}&\le C_0 \max_{j} (\omega_{jj}^*)^{-1/2} \bigg(\frac{|O|p\log(2np/\delta)}{n}\bigg)^{1/2}.\label{rhs3}
\end{align}
Here $C_0$ is an universal constant smaller than $4224$.
\end{theorem}

Several comments are in order. First of all, let us stress that the obtained guarantees are nonasymptotic: it is not required that
the sample size $n$ or another quantity tend to infinity for this result to be true. To the best of our knowledge, this is the
first\footnote{When this work was in preparation, the preprint \citep{LohTan2015} has been posted on arxiv that contains nonasymptotic results for another robust estimator of a multivariate Gaussian model. Detailed comparison of the results therein with the ours is provided below in the discussion on the previous work.} nonasymptotic result in robust estimation of a multivariate Gaussian model. Second, the value of the tuning parameter proposed
by this result is scale free, that is it does not depend on the magnitude of the unknown parameters of the model. Third, one can show
that the right-hand side expressions in Eq.~(\ref{rhs1})-(\ref{rhs3}) are minimax optimal up to logarithmic terms. Thus, the same estimator
of $\bfTheta^*$ is provably optimal for the three aforementioned norms. This remarkable property is due to the particular form of the
penalty used in the estimation procedure.

Let us switch now to results describing statistical properties of the estimator $\hat\bfOmega$ of the precision matrix. Unfortunately,
mathematical formulae we obtained as risk bounds for $\hat\bfOmega$ are not as compact and elegant as those of the last theorem.
Therefore, to improve their legibility, we opted for presenting the results in a more asymptotic form. Namely, we replace the condition
$40|O|p(13\log(2np/\delta)+2(1+M_{\bfE})^2)\le n-|O|$ by the following one $|O|p\log n\le c_0 n$, for some sufficiently small constant $c_0>0$, and
we do not provide explicit constants.

\begin{theorem}\label{thm:2}
Let assumptions (C1) and (C2) be satisfied and let $\lambda$ be as in \eqref{lambd}.
Then there exist universal constants $C,c_0>0$ and $n_0\in\mathbb N$ such that for $n\ge n_0$ and $|O|p\log n \le c_0 n$, the inequality
\begin{align}
{\|\hat\bfOmega-\bfOmega^*\|}_{2,2}&\le C\frac{\sigma_{\max}(\bfOmega^*)^2}{\sigma_{\min}(\bfOmega^*)}
\bigg\{M_{\bfE}\frac{|O|p\log n}{n}+\bigg(\frac{p^2\log n}{n}\bigg)^{1/2}\bigg\}\label{rhs4}
\end{align}
holds true with probability at least $1-5/n$.
\end{theorem}

This result tells us that in an asymptotic setting when all the three parameters $n$, $p$ and $|O|$ are allowed to tend to
infinity but so that $|O|p = o(n/\log n)$, the rate of convergence of the estimator $\hat\bfOmega$, measured in the
Frobenius norm, is $p(\frac{|O|}{n}+\frac1{n^{1/2}})$. This rate contains two components, $p/n^{1/2}$ and $p|O|/n$, each of which has a
clear interpretation. The rate $p/n^{1/2}$ comes from the fact that we are estimating $p^2$ entries of the matrix
$\bfOmega^*$ based on $n$ observations. This term is unavoidable if no additional assumption (such as the sparsity) is made;
it is the minimax rate of convergence in the outlier-free set-up. The second term, $p|O|/n$, originates from the fact that
the outlier matrix has $p|O|$ nonzero entries which need to be somehow estimated for making it possible to estimate the model
parameters. So, this term of the risk reflects the deterioration caused by the presence of outliers.

\section{Discussion}\label{sec:4}

\paragraph{Our bounds versus those of always zero estimator}

Given that the matrix $\bfTheta^*$ is defined as $\bfE^*$ divided by $\sqrt{n}$, one may wonder what is the advantage of our results
as compared to the risk bound of the trivial estimator $\hat\bfTheta^0$ all the entries of which are $0$. Clearly, the error of this
estimator measured in Frobenius norm is of the order $M_{\bfE}^2|O|p/n$. One may erroneously think that this bound is of the same
order as the one we obtained above for the convex programming based estimator. In contrast with this, the risk bound of our
estimator---although requires $M_{\bfE}^2|O|p/n$ to be bounded by some small constant---does not depend on $M_{\bfE}$. For instance, if
$M_{\bfE}= \frac1{12}(\frac{n}{|O|p})^{1/2}$, the trivial estimator will have a constant risk whereas the estimator $\hat\bfTheta$
will be consistent and rate optimal provided that $|O|p\log(n+p) = o(n)$.

Another important advantage of our estimator---inherent to its definition and reflected in the obtained risk bounds---is that its
squared error is proportional to the quantity $\max_{j\in p} (\phi_j^*){}^2$, where $(\phi_j^*)^2$  represents the conditional variance
of the $j$-th variable given all the others. In situations where the variables contain strong correlations, these conditional variances
are significantly smaller than the marginal variances of the variables.

\paragraph{What happens if some outliers have very large norms ?}
The risk bound established for our estimator requires the constant $M_{\bfE}$, measuring the order of magnitude of the Euclidean norm
of the outliers, to be not too large. This is not an artifact of our mathematical arguments, but an inherent limitation of our method.
We did some experiments on simulated data that confirmed that when $M_{\bfE}$ is large, our estimator behaves poorly. However, we believe
that this is not a serious limitation, since one can always pre-process the data by removing the observations that have atypically
large Euclidean norm.

\paragraph{Lower bounds}
It is possible to establish lower bounds that show that the rates of convergence of the risk bounds that appear in Theorem~\ref{thm:1}
are optimal up to logarithmic factors. Indeed, one can show that there exists a constant $c>0$ such that
\begin{align}
\inf_{\bar\bfTheta_n} \sup_{(\bfOmega^*,\bfTheta^*)} \bfE\big[\|\bar\bfTheta_n-\bfTheta^*\|_{q,q'}\big]
\ge c\bigg(\frac{p^{2/q}|O|^{2/q'}}{n}\bigg)^{1/2},
\qquad (q,q')\in\{(1,1);(1,2);(2,2)\},
\end{align}
where the $\inf$ is over all possible estimators $\bar\bfTheta_n$ while the $\sup$ is over all matrices
$\bfOmega^*,\bfTheta^*$ such that $\bfE^* = \sqrt{n} \bfTheta^*\diag(\bfOmega^*)(\bfOmega^*)^{-1}$ satisfies condition (C2).
This lower bound can be proved by lower bounding the $\sup$ over all possible precision matrices by the corresponding
expression for the identity precision matrix $\bfOmega^* = \bfI_p$. In this case, $\bfE^* = \sqrt{n}\,\bfTheta^*$
and we observe $\bfX^{(n)} = \bfTheta^*+ n^{-1/2}\bepsilon$, where $\bepsilon$ is a $n\times p$  matrix with iid standard
Gaussian entries. If we further lower bound the $\sup$ over all $|O|$-(row)sparse matrices $\bfTheta^*$ by the sup over
matrices whose rows $|O|+1,\ldots,n$ vanish, we get a simple Gaussian mean estimation problem for the entries
$\theta^*_{ij}$ with $i=1,\ldots,|O|$ and $j=1,\ldots,p$, under the condition $\max_{i,j}|\theta_{ij}^*|\le n^{-1/2}M_{\bfE}$.
It is well known that in this problem the individual entries $\theta_{ij}^*$ can not be estimated at a rate faster than
$n^{-1/2}$. This yields the result for $q=q'=1$. The corresponding upper bounds for $(q,q') = (2,1)$ and
$(q,q') = (2,2)$ readily follow from that of $(q,q')=(1,1)$ by a simple application of the Cauchy-Schwarz inequality.
Furthermore, very recently, the cases $(q,q')=(2,1)$ and $(q,q')=(2,2)$ have been thoroughly studied by \cite{KloppTsyb15}.
In particular, lower bounds including logarithmic terms have been established that prove that our estimator is minimax
rate optimal when $p/|O|$ is of the order $n^r$ for some $r\in(0,1)$.

\paragraph{$\epsilon$-contamination model and minimax optimality}

The estimator proposed in this work can be applied in the context of $\epsilon$-contamination model
often used in statistics for quantifying the performance of robust estimators. It corresponds to
assuming that each of $n$ rows of the data matrix $\bfX$ is given by $\bfX_i = (1-\epsilon_i)\bfY_i  + \epsilon_i\bfE_i$,
where $\epsilon_i\in \{0,1\}$ is a Bernoulli random variable with $\bfP(\epsilon_i=1)=\epsilon$, $\bfY_i\sim \mathcal N(\bmu^*,
\bfSigma^*)$ is as  before and $\bfE_i$ is randomly drawn from a distribution $Q$. The random variables $\epsilon_i$, $\bfY_i$
and $\bfE_i$ are independent and, perhaps the main difference with the model we considered above is that all $\bfE_i$'s
are drawn from the same distribution $Q$. One may wonder whether our procedure is minimax optimal in this
$\epsilon$-contamination model.

As proved in Theorems 3.1 and 3.2 of \citep{Gao2015}, the minimax rate
for estimating the covariance matrix  $\bfSigma^*$ in the squared operator norm is $\frac{p}{n}+ \epsilon^2$. In our
notation, the role of $\epsilon$  is played by $|O|/n$. Therefore, the aforementioned result from \citep{Gao2015}
suggests that one can estimate the precision matrix in the squared Frobenius norm with the rate $p(\frac{p}{n}+
\epsilon^2) = p(\frac{p}{n}+\frac{|O|^2}{n^2})$, where the factor $p$ comes from the fact that the square of the
Frobenius norm is upper bounded by $p$-times the operator norm. Recall that the rate provided by the upper bound
of Theorem~\ref{thm:2} is $p(\frac{p}{n}+\frac{|O|^2p}{n^2})$.

Therefore, the rate obtained by a direct application of Theorem~\ref{thm:2} is sub-optimal in the minimax sense for
the $\epsilon$-contamination model (when both the dimension and the number of outliers tend to infinity with the
sample size so that $|O|^2/n $ tends to infinity). We explain in Section~\ref{app:2} below the reason of
this sub-optimality and outline an approach for getting optimal rates, up to logarithmic factors. It is still
an open question whether the rate $p(\frac{p}{n}+\frac{|O|^2p}{n^2})$ is minimax optimal over the set
$\mathcal M(\underline\tau,\overline\tau,M_{\bfE})$ of matrices $(\bfSigma^*,\bfE^*)$ such that $\underline\tau\le
\sigma_{\min}(\bfSigma^*)\le \sigma_{\max}(\bfSigma^*)\le\overline\tau$ and $\bfE^*$ satisfies condition \textbf{(C2)}.
Theorem~\ref{thm:2} establishes that $p(\frac{p}{n}+\frac{|O|^2p}{n^2})$ is an upper bound for the minimax rate,
but the question of getting matching lower bound remains open.

\paragraph{Extensions to the case of large $p$}\label{sec:2}

In the case of large $p$, most ingredients of the proof used in moderate dimensional case remain valid after a suitable
adaptation. Perhaps the most important difference is in the definition of the dimension-reduction cone. In order to present it,
let $\calJ = \{J_j:j\in[p]\}$ be a collection of $p$ subsets of $[p]$--supports of each row of the precision matrix--for which
we use the notation $|\calJ| = \sum_{j=1}^p |J_j|$.  By a slight abuse of notation, we will write $\calJ^c$ for the collection
$\{J_j^c:j\in[p]\}$ and, for every $p\times p$ matrix $\bfA$, we define $\bfA_\calJ$ as the matrix obtained from $\bfA$ by zeroing
all the elements $\bfA_{i,j}$ such that $i\not\in J_j$. Let $O$ be the subset of $[n]$ corresponding to the outliers. We define the
dimension reduction cone
\begin{align*}
\mathscr C_{\calJ, O}(c,\gamma) \triangleq \Big\{ \bfDelta \in \R^{(p+n)\times p} :
\gamma{\|\bfDelta^\bfB_{\calJ^c}\|}_{1,1} + {\|\bfDelta^\bfTheta_{O^c,\bullet}\|}_{2,1}
&\le c \big(\gamma{\|\bfDelta^\bfB_{\calJ}\|}_{1,1} +  {\|\bfDelta^\bfTheta_{O,\bullet}\|}_{2,1}\big) \Big\} ,
\end{align*}
for $c > 1$ and $\gamma>0$, where $\bfDelta^\bfB = \bfDelta_{1:p,\bullet}$ and $\bfDelta^\bfTheta = \bfDelta_{(p+1):(p+n),\bullet}$.
For a constant $\kappa>0$, let us introduce the matrix $\bfM = [\bfX^{(n)};-\bfI_n]$ and the event
\begin{align}
 \mathcal E_\kappa = \bigg\{{\|\bfM \bfDelta\|}_F^2 &\ge \kappa \bigg( \frac{{\|\bfDelta^\bfB_{\calJ}\|}_{1,1}^2}{|\calJ|}\bigg)\bigvee\bigg(\frac{{\|\bfDelta^\bfTheta_{O,\bullet}\|}_{2,1}^2}{|O|}\bigg) \quad \text{for all}\quad \bfDelta \in \mathscr C_{\calJ,O}(2,1)\bigg\}. \label{eq:mcc}
\end{align}
This event corresponds to the situations where the matrix $\bfM$ satisfies the (matrix) compatibility condition.
To simplify the statement of the result, we assume that all the diagonal entries of the covariance matrix $\bfSigma^*$ are
equal to one. Note that this assumption can be approached by dividing the columns of $\bfX$ by the corresponding robust
estimators of their standard deviation.

\begin{theorem}\label{thm:3}
Let $\calJ$ and $O$ be such that $\bfB^*_{\calJ^c} = 0$ and $\bfTheta^*_{O^c\!,\bullet} = 0$. Choose $\gamma=1$
and $\delta\in(0,1)$ such that $n\ge |O|+16\log(2p/\delta)$ and choose
\begin{equation}\label{lambd1}
\lambda = 6 \bigg(\frac{\log(2np/\delta)}{n-|O|}\bigg)^{1/2}.
\end{equation}
If\ $4\lambda(|\calJ|^{1/2}+ |O|^{1/2}) < \kappa^{1/2}$ holds, then there exists an event $\mathcal E_0$ of probability at least
$1-2\delta$ such that in $\mathcal E_\kappa\cap \mathcal E_0$, we have
 \begin{align}
{\|\hat\bfB-\bfB^*\|}_{1,1}+{\|\hat\bfTheta-\bfTheta^*-\bxi_{O,\bullet}\|}_{2,1}
&\le \frac{C_1}{\kappa}\max_{j\in[p]} (\omega^*_{jj})^{-1/2} \big(|\calJ| + |O|\big)\bigg(\frac{\log(2np/\delta)}{n-|O|}\bigg)^{1/2} \label{eq:5q}
 \end{align}
 with $C_1 \le 900$.
\end{theorem}

The proof of this theorem follows the same scheme as the one of Theorem~\ref{thm:1}, that is the main reason of placing this
proof in the supplementary material. We will not comment this result too much because we find it incomplete at this stage. Indeed,
the main conclusion of the theorem is formulated as a risk bound that holds in an event close to $\mathcal E_\kappa$.
Unfortunately, we are not able now to provide a theoretical evaluation of $\bfP(\mathcal E_\kappa)$. We believe however that this
probability is close to one, since the matrix $\bfM$ is composed of two matrices $\bfX^{(n)}$ and $-\bfI_n$ that have weakly
correlated columns and each of these matrices satisfy the restricted eigenvalues condition. We hope that we will be able to make
this rigorous in near future. Note also that this result tells us that one gets the optimal rate (up to logarithmic factors) of estimating $\bfB^*$ in $\ell_1$-norm if the number of outliers is at most of the same order as the sparsity of the precision matrix.

\paragraph{Other related work}
In recent years, several methodological contributions have been made to the problem of robust estimation in multivariate
Gaussian models under various kinds of contamination models. For  instance, \cite{WangLin2014} have proposed a
group-lasso type strategy in the context of errors-in-variables with a pre-specified group structure on the set of
covariates whereas \cite{HiroseFujisawa2015} have introduced the method $\gamma$-lasso, a robust sparse estimation
procedure of the inverse covariance matrix based on the $\gamma$-divergence. Under cell-wise contamination model,
\cite{OllererCroux2015} and \cite{TarrMullerWeber2015} proposed to estimate the precision matrix by using either
the graphical lasso \citep{FriedmanHT,dAsprBanerjeeEG} or the CLIME estimator \citep{CaiLiuLuo} in conjunction with a robust
estimator of the covariance matrix. While~\citep{TarrMullerWeber2015} have mainly focused on the methodological aspects,
\citep{OllererCroux2015} carried out a breakdown analysis. Risk bounds on the statistical error of this procedure has
been established by \cite{LohTan2015}. They have shown that the element-wise squared error when estimating the precision
matrix $\bfOmega^*$ is of the order $\|\bfOmega^*\|_{1,\infty}^2\, \big(\frac{p}{n}+\frac{|O|^2}{n^2}\big)$. This result
is particularly appealing for very sparse precision matrices having small $\ell_{1,\infty}$ norm. However, in moderate dimensional
situations where the precision matrix is not necessarily sparse, the  term $\|\bfOmega^*\|_{1,\infty}^2$ is generally
proportional to $p\sigma_{\max}(\bfOmega^*)^2$ and the resulting upper bound is very likely to be sub-optimal. If we apply
this result for assessing the quality of estimation in the squared Frobenius norm, we get an upper bound of the order
$p^2\big(\frac{p}{n}+\frac{|O|^2}{n^2}\big)$, whereas our result provides an upper bound of the order
$p\big(\frac{p}{n}+\frac{|O|^2p}{n^2}\big)$.  Furthermore, the results in \citep{LohTan2015} require the tuning parameter $\lambda$
to be larger than an expression that involves the proportion of the outliers and  the $\ell_{1,\infty}$ norm of the matrix
$\bfOmega^*$. This quantities are rarely available in practice and their estimation is often a hard problem. Finally,
in the context of robust estimation of large matrices, let us also mention the recent work \citep{Klopp15}, proposing
a robust method of matrix completion and establishing sharp risk bounds on its statistical error.

\section{Technical results and proofs}\label{sec:6}

This section contains the proofs of all the mathematical claims of the paper, except Theorem~\ref{thm:3}, the proof of which is placed
in the supplementary material. The section is split into three parts. The first part contains the proof of Theorem~\ref{thm:1}, up to
some technical lemmas characterizing the order of magnitude of the stochastic terms. The proof of Theorem~\ref{thm:2} is presented in
the second part, while the third part contains the aforementioned lemmas on the tail behaviour of random quantities appearing in the
proofs.

To ease notation, we define the projection matrix $\bfZ = \bfI_n - \bfX\big(\bfX^\top \bfX\big)^\dag \bfX^\top$.

\subsection{Risk bounds for outlier estimation}

In this subsection, we provide a proof of Theorem~\ref{thm:1}, which contains perhaps the most original mathematical arguments of this work.
Prior to diving into low-level technical arguments, let us provide a high-level overview of the proof. We can split it into
four steps as follows:
\begin{description}
\item[Step 1:] We check that if
  \begin{align} \label{eq:7}
  \lambda \ge 3 \max_{i\in[n]} \bigg(\sum_{j\in [p]} \frac{(\bfZ^j_{i,\bullet}\bepsilon_{\bullet,j})^2}{{\|{\displaystyle\bfZ^j} \bepsilon_{\bullet,j}\|}_2^2}\bigg)^{1/2}
  \end{align}
	then the vector $\hat\bfDelta^\bfTheta = \hat\bfTheta-\bfTheta^*$ belongs to the dimension-reduction cone
  \begin{align}\label{eqn:cone}
  {\|\hat\bfDelta^\bfTheta_{O^c,\bullet}\|}_{2,1} \le 2 {\|\hat\bfDelta^\bfTheta_{O,\bullet}\|}_{2,1}.
  \end{align}
\item[Step 2:] Using the Karush-Kuhn-Tucker conditions, we establish the bound
\begin{align}\label{eqn:KKT}
{\|\bfZ \hat\bfDelta^\bfTheta\|}_{2,2}^2 \le \frac{14\lambda}{3}   {\|\bxi^\top\|}_{2,\infty} {\|\hat\bfDelta^\bfTheta\|}_{2,1} + \big( \lambda {\|\hat\bfDelta^\bfTheta\|}_{2,1} \big)^2
\end{align}
for $\lambda$ satisfying \eqref{eq:7}.
\item[Step 3:] Combining the two previous steps and using notation $\alpha:={\|\bfI_n-\bfZ\|}_{\infty,\infty}$,
we obtain
   \begin{align}\label{eq:8b}
    {\|\hat\bfDelta^\bfTheta\|}_{2,2}\le 140\lambda   {\|\bxi^\top\|}_{2,\infty}|O|^{1/2} \qquad \text{and} \qquad {\|\hat\bfDelta^\bfTheta\|}_{2,1} \le 520\lambda   {\|\bxi^\top\|}_{2,\infty}|O|,
   \end{align}
provided that $|O|( \lambda^2+\alpha)<1/10$.
\item[Step 4:]  We conclude by establishing deterministic bounds on the random variables that appear in expressions \eqref{eq:7} and \eqref{eq:8b},
 as well as on $\alpha$.
\end{description}
The proofs of Steps 1 and 4 are, up to some additional technicalities, similar to those for square-root lasso. Steps 2 and 3 contain
more original ingredients. The detailed proofs of all these steps are given below.

  For every $c>0$ and $O\subset[n]$, we define the cone
  \begin{align*}
   \mathscr C_{O}(c) \triangleq \Big\{ \bfDelta \in \R^{n\times p} :
  {\|\bfDelta^\bfTheta_{O^c,\bullet}\|}_{2,1} \le c {\|\bfDelta^\bfTheta_{O,\bullet}\|}_{2,1} \Big\}.
  \end{align*}

  \begin{lemma}\label{step1}
   If, for some constant $c>1$, the penalty level $\lambda$ satisfies the condition
   \begin{align}
    \lambda \ge \frac{c+1}{c-1} \max_{i\in[n]} \bigg(\sum_{j\in [p]} \frac{(\bfZ^j_{i,\bullet}\bepsilon_{\bullet,j})^2}{{\|{\displaystyle\bfZ^j} \bepsilon_{\bullet,j}\|}_2^2}\bigg)^{1/2}, \label{eq:K}
   \end{align}
   then the matrix $\hat\bfDelta^\bfTheta$ belongs to the cone $\mathscr C_{O}(c)$.
  \end{lemma}

  \begin{proof}
%    This proof follows the same sketch as Proposition \ref{prop:1}.
   The definition of $\hat\bfTheta$ by optimization problem \eqref{eq:5L} immediately leads to
   \begin{align}
    \lambda \big({\|\hat\bfTheta\|}_{2,1} - {\|\bfTheta^*\|}_{2,1}\big) \le \sum_{j\in[p]}
    \big(\|\bfZ^j(\bfX^{(n)}_{\bullet,j}-\bfTheta^*_{\bullet,j})\big\|_2-
    \|\bfZ^j(\bfX^{(n)}_{\bullet,j}-\hat\bfTheta_{\bullet,j})\big\|_2\big). \label{eq:O}
   \end{align}
   We use the inequality $\|a\|_2-\|b\|_2\le (a-b)^\top a/\|a\|_2$ which ensues from the Cauchy-Schwarz inequality and is true for any pair of vectors $(a,b)$, here with
   $a=\bfZ^j(\bfX^{(n)}_{\bullet,j}-\bfTheta^*_{\bullet,j})$ and $b=\bfZ^j(\bfX^{(n)}_{\bullet,j}-\hat\bfTheta_{\bullet,j})$.
   Clearly, we have $a-b = \bfZ^j \hat\bfDelta^\bfTheta_{\bullet,j}$ and $a = \bfZ^j \bxi_{\bullet,j}$. Hence, we obtain
   \begin{align*}
    {\big\|\bfZ^j(\bfX^{(n)}_{\bullet,j}-\bfTheta^*_{\bullet,j})\big\|}_2 - {\big\|\bfZ^j(\bfX^{(n)}_{\bullet,j}-\hat\bfTheta_{\bullet,j})\big\|}_2 &\le (\bfZ^j \hat\bfDelta^\bfTheta_{\bullet,j})^\top \frac{\bfZ^j \bxi_{\bullet,j}}{{\big\|\bfZ^j \bxi_{\bullet,j}\big\|}_2} %\\
%    &= {\displaystyle\hat\bfDelta^\bfTheta_{\bullet,j}}^\top \frac{\bfZ^j \bxi_{\bullet,j}}{{\big\|\bfZ^j \bxi_{\bullet,j}\big\|}_2} \\&
		=\sum_{i=1}^n \hat\bfDelta^\bfTheta_{i,j} \frac{\bfZ^j_{i,\bullet} \bxi_{\bullet,j}}{{\big\|\bfZ^j \bxi_{\bullet,j}\big\|}_2}.
   \end{align*}
  Then summing on $j\in[p]$ and applying the Cauchy-Schwarz inequality, we get
   \begin{align*}
    \sum_{j\in[p]}\|\bfZ^j(\bfX^{(n)}_{\bullet,j}-\bfTheta^*_{\bullet,j})\big\|_2-\|\bfZ^j(\bfX^{(n)}_{\bullet,j}-\hat\bfTheta_{\bullet,j})\big\|_2
    &\le \sum_{i=1}^n {\|\hat\bfDelta^\bfTheta_{i,\bullet}\|}_2 \bigg(\sum_{j=1}^p \frac{(\bfZ^j_{i,\bullet} \bxi_{\bullet,j})^2}{{\big\|\bfZ^j \bxi_{\bullet,j}\big\|}_2^2}\bigg)^\frac12 .
   \end{align*}
  This inequality, in conjunction with Eq.~\eqref{eq:O} and the obvious inequality ${\|\hat\bfTheta\|}_{2,1} - {\|\bfTheta^*\|}_{2,1} \ge {\|\hat\bfDelta^\bfTheta_{O^c,\bullet}\|}_{2,1}-{\|\hat\bfDelta^\bfTheta_{O,\bullet}\|}_{2,1}$ leads to
  \begin{align*}
   \lambda \big({\|\hat\bfDelta^\bfTheta_{O^c,\bullet}\|}_{2,1}-{\|\hat\bfDelta^\bfTheta_{O,\bullet}\|}_{2,1}\big) &\le {\|\hat\bfDelta^\bfTheta\|}_{2,1} \max_{i\in[n]} \bigg(\sum_{j=1}^p \frac{(\bfZ^j_{i,\bullet} \bxi_{\bullet,j})^2}{{\big\|\bfZ^j \bxi_{\bullet,j}\big\|}_2^2}\bigg)^\frac12 \\
   &\le \lambda \frac{c-1}{c+1} \Big({\|\hat\bfDelta^\bfTheta_{O,\bullet}\|}_{2,1} + {\|\hat\bfDelta^\bfTheta_{O^c,\bullet}\|}_{2,1}\Big)  ,
  \end{align*}
  where the last line follows from condition \eqref{eq:K}. In conclusion, we get ${\|\hat\bfDelta^\bfTheta_{O^c,\bullet}\|}_{2,1} \le c {\|\hat\bfDelta^\bfTheta_{O,\bullet}\|}_{2,1}$, which coincides with the claim of the lemma.
\end{proof}

The second step will be split into several lemmas, whereas the final conclusion is presented below in Lemma~\ref{step2}.

\begin{lemma} \label{lem:A}
 Let us introduce the vectors $\hat\bxi_{\bullet,j} = \bfZ^j(\bfX^{(n)}_{\bullet,j}-\hat\bfTheta_{\bullet,j})$, $j\in[p]$.
There exists a $n\times p$ matrix $\bfV$ such that
\begin{align}\label{eq:A1}
\|\bfV_{i,\bullet}\|_2\le 1,\qquad
\bfV_{i,\bullet}^\top\hat\bfTheta_{i,\bullet} = \|\hat\bfTheta_{i,\bullet}\|_2,\qquad\forall\,i\in[n],
\end{align}
and, for every $j\in[p]$, the following relation holds
\begin{align} \label{eq:D}
 {\|\bfZ^j \hat\bfDelta^\bfTheta_{\bullet,j}\|}_2^2 = \bxi_{\bullet,j}^\top \bfZ^j \hat\bfDelta^\bfTheta_{\bullet,j} -\lambda {\|\hat\bxi_{\bullet,j}\|}_2 \bfV_{\bullet,j}^\top \hat\bfDelta^\bfTheta_{\bullet,j} .
\end{align}
\end{lemma}

\begin{proof}
Let us first consider the case $\hat\bxi_{\bullet,j} \neq 0$. It is helpful to introduce the functions
$g_1(\bfTheta) = \sum_{j=1}^p{\big\|\bfZ^j (\bfX^{(n)}_{\bullet,j} - \bfTheta_{\bullet,j}) \big\|}_2$
and $g_2(\bfTheta)=\sum_{i=1}^n{\|\bfTheta_{i,\bullet}\|}_2$. The Karush-Kuhn-Tucker conditions imply
that there exist two matrices $\bfU$ and $\bfV$ in $\R^{n\times p}$ satisfying
$\bfU\in \partial_{\bfTheta}g_1(\hat\bfTheta)$, $\bfV\in \partial_{\bfTheta}g_2(\hat\bfTheta)$ and
$\bfU+\lambda\bfV =0$. For every $j\in[p]$, let $\bfu_j$ and $\bfv_j$ be the $j$th column of $\bfU$
and $\bfV$, respectively, so that  $\bfu_j+\lambda\bfv_j=0$ for every $j\in[p]$. With the assumption
that ${\|\hat\bxi_{\bullet,j}\|}_2 > 0$, $\bfu_j$ is a differential and
$\bfu_j = ({\bfZ^j}^\top \bfZ^j \hat\bfTheta_{\bullet,j} -  {\bfZ^j}^\top \bfX^{(n)}_{\bullet,j}) /
{\|\hat\bxi_{\bullet,j}\|}_2$. Thus ${\bfZ^j}^\top \bfZ^j = \bfZ^j$ leads to
$\bfu_j = \bfZ^j (\hat\bfTheta_{\bullet,j} - \bfX^{(n)}_{\bullet,j})/ {\|\hat\bxi_{\bullet,j}\|}_2$.
Hence, we deduce that $\bfZ^j \hat\bfTheta_{\bullet,j} - \bfZ^j \bfX^{(n)}_{\bullet,j} +
\lambda\bfv_j {\|\hat\bxi_{\bullet,j}\|}_2=0$. Furthermore, as $\bfX^{(n)}_{\bullet,j} = -
\bfX^{(n)}_{\bullet,j^c} \bfB^*_{j^c,j} + \bfTheta^*_{\bullet,j} + \bxi_{\bullet,j}$ and $\bfZ^j$ is
the projector onto the subspace orthogonal to $\bfX^{(n)}_{\bullet,j^c}$, it follows that
\begin{align}\label{eq:6L}
 \bfZ^j \bfX^{(n)}_{\bullet,j}  &= \bfZ^j \bfTheta^*_{\bullet,j} + \bfZ^j \bxi_{\bullet,j}.
\end{align}
This yields $\bfZ^j \hat\bfDelta^\bfTheta_{\bullet,j} - \bfZ^j \bxi_{\bullet,j} +\lambda{\|\hat\bxi_{\bullet,j}\|}_2 \bfv_j = 0 $
where $\hat\bfDelta^\bfTheta = \hat\bfTheta-\bfTheta^*$.
Finally, taking the scalar product of both sides with $\hat\bfDelta^\bfTheta_{\bullet,j}$, we get
\begin{align*}
 (\hat\bfDelta^\bfTheta_{\bullet,j})^\top \bfZ^j \hat\bfDelta^\bfTheta_{\bullet,j} -
\bxi_{\bullet,j}^\top\bfZ^j \hat\bfDelta^\bfTheta_{\bullet,j} +
\lambda{\|\hat\bxi_{\bullet,j}\|}_2 \bfv_j^\top\hat\bfDelta^\bfTheta_{\bullet,j}  = 0 .
\end{align*}
Since $\bfv_j = \bfV_{\bullet,j}$, this completes the proof of \eqref{eq:D}. To check relation (\ref{eq:A1}), it suffices to remark that
$\bfV_{i,\bullet}$ belongs to the sub-differential of the Euclidean norm $\|\bfTheta_{i,\bullet}\|_2$ evaluated at $\hat\bfTheta$.

Let us now consider the case $\hat\bxi_{\bullet,j} = 0$. This can be equivalently written as $\bfZ^j(\bfX^{(n)}_{\bullet,j}-
\hat\bfTheta_{\bullet,j})=0$. In view of Eq.\ \eqref{eq:6L}, we get $\bfZ^j\hat\bfDelta^{\bfTheta}_{\bullet,j}=\bfZ^j\bxi_{\bullet,j}$.
Taking the scalar product of both sides with $\hat\bfDelta^{\bfTheta}_{\bullet,j}$ and using the fact that $\bfZ^j$ is idempotent,
we get relation \eqref{eq:D}.
\end{proof}

\begin{lemma}\label{lem:aa}
Let $R,A,B$ be arbitrary real numbers satisfying the inequality
$R^2\le A+BR$. Then, the inequality $R^2\le 2A+B^2$ holds true.
\end{lemma}

\begin{proof}
The inequality $R^2\le A+BR$ is equivalent to $(2R-B)^2\le 4A+B^2$. This entails that
$|2R-B|\le \sqrt{4A+B^2}$ and, therefore, $2R\le B+\sqrt{4A+B^2}$. We get the desired result
by taking the square of both sides and using the inequality $(a+b)^2\le 2a^2+2b^2$.
\end{proof}

\begin{lemma} \label{lem:D}
 Equation \eqref{eq:D} implies that
 \begin{align*}
 {\|\bfZ^j \hat\bfDelta^\bfTheta_{\bullet,j}\|}_2^2 &\le 2 \bxi_{\bullet,j}^\top \bfZ^j \hat\bfDelta^\bfTheta_{\bullet,j}
 -2\lambda{\|\bfZ^j\bxi_{\bullet,j}\|}_2 \bfV_{\bullet,j}^\top \hat\bfDelta^\bfTheta_{\bullet,j}+ (\lambda  \bfV_{\bullet,j}^\top \hat\bfDelta^\bfTheta_{\bullet,j})^2.
\end{align*}
\end{lemma}

\begin{proof}
 According to Eq.\ \eqref{eq:6L}, we have  $\bfZ^j (\bfX^{(n)}_{\bullet,j}-\bfTheta^*_{\bullet,j}) = \bfZ^j \bxi_{\bullet,j}$.
Therefore, from the definition of the estimated residuals $\hat\bxi_{\bullet,j}$ we infer that $\bfZ^j\bxi_{\bullet,j} - \hat\bxi_{\bullet,j}=\bfZ^j
\hat\bfDelta^\bfTheta_{\bullet,j}$, which implies the inequality
\begin{align*}
 \big|{{\|\bfZ^j\bxi_{\bullet,j}\|}_2-\|\hat\bxi_{\bullet,j}\|}_2\big| &\le {\|\bfZ^j\bxi_{\bullet,j}-\hat\bxi_{\bullet,j}\|}_2
 ={\|\bfZ^j \hat\bfDelta^\bfTheta_{\bullet,j}\|}_2.
\end{align*}
Combining this bound with equation \eqref{eq:D} of Lemma \ref{lem:A}, we obtain
\begin{align*}
 {\|\bfZ^j \hat\bfDelta^\bfTheta_{\bullet,j}\|}_2^2
    &= \bxi_{\bullet,j}^\top \bfZ^j \hat\bfDelta^\bfTheta_{\bullet,j} -\lambda {\|\bfZ^j\bxi_{\bullet,j}\|}_2 \bfV_{\bullet,j}^\top \hat\bfDelta^\bfTheta_{\bullet,j} +\lambda \big({\|\bfZ^j\bxi_{\bullet,j}\|}_2 -{\|\hat\bxi_{\bullet,j}\|}_2) \bfV_{\bullet,j}^\top \hat\bfDelta^\bfTheta_{\bullet,j}\\
    &\le \bxi_{\bullet,j}^\top \bfZ^j \hat\bfDelta^\bfTheta_{\bullet,j} -\lambda{\|\bfZ^j\bxi_{\bullet,j}\|}_2 \bfV_{\bullet,j}^\top \hat\bfDelta^\bfTheta_{\bullet,j}
      +\lambda |\bfV_{\bullet,j}^\top \hat\bfDelta^\bfTheta_{\bullet,j}|\cdot {\|\bfZ^j \hat\bfDelta^\bfTheta_{\bullet,j}\|}_2.
\end{align*}
We conclude using Lemma \ref{lem:aa} with $R = {\|\bfZ^j \hat\bfDelta^\bfTheta_{\bullet,j}\|}_2$.
\end{proof}

\begin{lemma} \label{lem:C}
 Assuming that $\lambda \ge \frac{c+1}{c-1} \displaystyle\max_{i\in[n]} \bigg(\sum_{j\in [p]} \frac{(\bfZ^j_{i,\bullet}\bepsilon_{\bullet,j})^2}{{\|{\displaystyle\bfZ^j} \bepsilon_{\bullet,j}\|}_2^2}\bigg)^{1/2}$, it holds
 \begin{align*}
  \sum_{j=1}^p \bxi_{\bullet,j}^\top \bfZ^j \hat\bfDelta^\bfTheta_{\bullet,j} &\le \lambda \frac{c-1}{c+1} {\|\hat\bfDelta^\bfTheta\|}_{2,1} \max_{j\in[p]} {\|\bxi_{\bullet,j}\|}_2 .
 \end{align*}
\end{lemma}

\begin{proof}
 We have
 \begin{align*}
  \sum_{j=1}^p \bxi_{\bullet,j}^\top \bfZ^j \hat\bfDelta^\bfTheta_{\bullet,j}
			&= \sum_{i=1}^n \sum_{j=1}^p (\bfZ^j \bxi_{\bullet,j})_i \hat\bfDelta^\bfTheta_{i,j}
			\le \max_{j\in[p]} {\|\bfZ^j \bxi_{\bullet,j}\|}_2 \sum_{i=1}^n \sum_{j=1}^p \frac{|(\bfZ^j \bxi_{\bullet,j})_i|}
			{{\|\bfZ^j \bxi_{\bullet,j}\|}_2}|\hat\bfDelta^\bfTheta_{i,j} |.
 \end{align*}
 Thus by the Cauchy-Schwarz inequality and the assumption of the lemma,
 \begin{align*}
  \sum_{j=1}^p \bxi_{\bullet,j}^\top \bfZ^j \hat\bfDelta^\bfTheta_{\bullet,j} &\le \max_{j\in[p]} {\|\bfZ^j \bxi_{\bullet,j}\|}_2 \sum_{i=1}^n {\|\hat\bfDelta^\bfTheta_{i,\bullet}\|}_2 \bigg(\sum_{j\in [p]} \frac{(\bfZ^j_{i,\bullet}\bxi_{\bullet,j})^2}{{\|{\displaystyle\bfZ^j} \bxi_{\bullet,j}\|}_2^2}\bigg)^{1/2} \\
  &\le \lambda \frac{c-1}{c+1} {\|\hat\bfDelta^\bfTheta\|}_{2,1} \max_{j\in[p]} {\|\bfZ^j \bxi_{\bullet,j}\|}_2 .
 \end{align*}
 Moreover, as the operator norm associated with the Euclidean norm is the spectral norm, it holds that ${\|\bfZ^j \bxi_{\bullet,j}\|}_2 \le {\|\bfZ^j\|}_2 {\|\bxi_{\bullet,j}\|}_2$. Then, as $\bfZ^j$ is a projection matrix, ${\|\bfZ^j\|}_2 = 1$ and  ${\|\bfZ^j \bxi_{\bullet,j}\|}_2 \le {\|\bxi_{\bullet,j}\|}_2$.
 The claimed result follows.
\end{proof}

  \begin{lemma}\label{step2}
   If conditions \eqref{eq:K} and \eqref{eq:D} hold, then
   \begin{align} \label{eq:B}
    \sum_{j=1}^p {\|\bfZ^j \hat\bfDelta^\bfTheta_{\bullet,j}\|}_2^2 &\le 2 \lambda  {\|\bxi^\top\|}_{2,\infty} {\|\hat\bfDelta^\bfTheta\|}_{2,1} \Big(\frac{c-1}{c+1} + 2\Big) + \Big( \lambda {\|\hat\bfDelta^\bfTheta\|}_{2,1} \Big)^2 .
   \end{align}
  \end{lemma}

  \begin{proof}
  We first note that $\|\bfV_{i,\bullet}\|_2\le 1$ yields
  \begin{align*}
    \sum_{j=1}^p |\bfV_{\bullet,j}^\top\hat\bfDelta^\bfTheta_{\bullet,j}| \le {\|\hat\bfDelta^\bfTheta\|}_{2,1}  \quad\text{and}\quad
    \sum_{j=1}^p (\lambda \bfV_{\bullet,j}^\top\hat\bfDelta^\bfTheta_{\bullet,j})^2 \le (\lambda{\|\hat\bfDelta^\bfTheta\|}_{2,1})^2 .
  \end{align*}
  Thus, using relation \eqref{eq:D} and Lemma~\ref{lem:D}, we arrive at
  \begin{align*}
   \sum_{j=1}^p {\|\bfZ^j \hat\bfDelta^\bfTheta_{\bullet,j}\|}_2^2 &\le 2 \sum_{j=1}^p \bxi_{\bullet,j}^\top \bfZ^j \hat\bfDelta^\bfTheta_{\bullet,j}
   +2\lambda \sum_{j=1}^p {\|\bfZ^j\bxi_{\bullet,j}\|}_2 |\bfV_{\bullet,j}^\top \hat\bfDelta^\bfTheta_{\bullet,j}| + \sum_{j=1}^p (\lambda \bfV_{\bullet,j}^\top\hat\bfDelta^\bfTheta_{\bullet,j})^2  \\
   &\le 2 \sum_{j=1}^p \bxi_{\bullet,j}^\top \bfZ^j \hat\bfDelta^\bfTheta_{\bullet,j}
   +2\lambda \max_{j\in[p]} {\|\bfZ^j\bxi_{\bullet,j}\|}_2 \sum_{j=1}^p |\bfV_{\bullet,j}^\top \hat\bfDelta^\bfTheta_{\bullet,j}| + (\lambda{\|\hat\bfDelta^\bfTheta\|}_{2,1})^2 \\
   &\le 2 \sum_{j=1}^p \bxi_{\bullet,j}^\top \bfZ^j \hat\bfDelta^\bfTheta_{\bullet,j}
   +2\lambda \max_{j\in[p]} {\|\bxi_{\bullet,j}\|}_2 {\|\hat\bfDelta^\bfTheta\|}_{2,1} + (\lambda{\|\hat\bfDelta^\bfTheta\|}_{2,1})^2.
  \end{align*}
  The combination of the latter with Lemma~\ref{lem:C} implies inequality \eqref{eq:B}.
  \end{proof}

Note that $\bfZ$ and $\bfZ^j$ are two orthogonal projection matrices on nested subspaces of dimensions $n-p$ and $n-p+1$, respectively.
Hence, for any $j \in [p]$, ${\|\bfZ \hat\bfDelta^\bfTheta_{\bullet,j}\|}_2 \le {\|\bfZ^j \hat\bfDelta^\bfTheta_{\bullet,j}\|}_2$. Using this
inequality to lower bound the left-hand side of Eq.\ \eqref{eq:B} and choosing $c=2$, we get inequality \eqref{eqn:KKT} of Step 2. We are now in a
position to carry out Step 3.

  \begin{prop}\label{prop:7}
   If the penalty level $\lambda$ satisfies the condition \eqref{eq:7} and $|O|( \lambda^2+\alpha)<1/10$, then
   \begin{align}\label{eq:8}
    {\|\hat\bfDelta^\bfTheta\|}_{2,2}\le 140\lambda   {\|\bxi^\top\|}_{2,\infty}|O|^{1/2} \qquad
		\text{and} \qquad p^{-1/2}{\|\hat\bfDelta^\bfTheta\|}_{1,1}\le {\|\hat\bfDelta^\bfTheta\|}_{2,1} \le 520\lambda
		{\|\bxi^\top\|}_{2,\infty}|O|,
   \end{align}
   where $\alpha:={\|\bfI_n-\bfZ\|}_{\infty,\infty}$.
  \end{prop}

  \begin{proof}
    In the few lines that follow, we write $\bfX$ instead of $\bfX^{(n)}$ and $\hat\bfDelta$ instead of $\hat\bfDelta^\bfTheta$. Simple algebra yields
    \begin{align*}
    {\|(\bfI_n-\bfZ) \hat\bfDelta\|}_{2,2}^2
	  & =\trace\big((\bfI_n-\bfZ)\hat\bfDelta((\bfI_n-\bfZ)\hat\bfDelta)^\top\big).
    \end{align*}
    Using the facts that $\trace(\bfA \bfB) = \trace(\bfB\bfA)$ (whenever the matrix products are well defined), $\trace(\bfA\bfB)\le \|\bfA\|_{\infty,\infty}\|\bfB\|_{1,1}$ and
    $\|\bfA\bfA^\top\|_{q,q}\le \|\bfA\|_{2,q}^2$, for any $q\in[1,\infty]$, (the last one is a simple consequence of the Cauchy-Schwarz inequality) we get
    \begin{align*}
    {\|(\bfI_n-\bfZ) \hat\bfDelta\|}_{2,2}^2
	  =\trace\big((\bfI_n-\bfZ)\hat\bfDelta\hat\bfDelta^\top\big)
	  \le {\|\bfI_n-\bfZ\|}_{\infty,\infty}\cdot {\|\hat\bfDelta\hat\bfDelta^\top\|}_{1,1}
	  \le {\|\bfI_n-\bfZ\|}_{\infty,\infty}\cdot {\|\hat\bfDelta\|}_{2,1}^2.
    \end{align*}
    Adding the last inequality to Eq.\ \eqref{eqn:KKT} of Step 2 and using the Pythagorean theorem, we get
    \begin{align}\label{eqn:2}
    {\|\hat\bfDelta\|}_{2,2}^2 \le \frac{14\lambda}{3}   {\|\bxi^\top\|}_{2,\infty} {\|\hat\bfDelta\|}_{2,1} + ( \lambda^2+\alpha) {\|\hat\bfDelta\|}_{2,1}^2 .
    \end{align}
    Since according to Step 1 we have $\hat\bfDelta^\bfTheta\in \mathscr C_O(c)$, we infer that
    \begin{align*}
    {\|\hat\bfDelta\|}_{2,2}^2 \le 14\lambda   {\|\bxi^\top\|}_{2,\infty} {\|\hat\bfDelta_{O,\bullet}\|}_{2,1} + 9( \lambda^2+\alpha) {\|\hat\bfDelta_{O,\bullet}\|}_{2,1}^2.
    \end{align*}
    Finally, using the Cauchy-Schwarz inequality, we have ${\|\hat\bfDelta_{O,\bullet}\|}_{2,1}^2\le |O|\cdot {\|\hat\bfDelta_{O,\bullet}\|}_{2,2}^2$, which leads to
    \begin{align*}
    {\|\hat\bfDelta\|}_{2,2}^2 \le 14\lambda   {\|\bxi^\top\|}_{2,\infty}|O|^{1/2} {\|\hat\bfDelta_{O,\bullet}\|}_{2,2} + 9|O|( \lambda^2+\alpha) {\|\hat\bfDelta_{O,\bullet}\|}_{2,2}^2.
    \end{align*}
    Since the last norm in the right-hand side is bounded from above by ${\|\hat\bfDelta\|}_{2,2}$, we get
    \begin{align*}
    {\|\hat\bfDelta\|}_{2,2}^2\le 14\lambda   {\|\bxi^\top\|}_{2,\infty}|O|^{1/2} {\|\hat\bfDelta\|}_{2,2} + 9|O|( \lambda^2+\alpha) {\|\hat\bfDelta\|}_{2,2}^2.
    \end{align*}
    This implies that either ${\|\hat\bfDelta\|}_{2,2}=0$ or
    \begin{align}\label{eqn:3}
    {\|\hat\bfDelta\|}_{2,2}\le \frac{14\lambda   {\|\bxi^\top\|}_{2,\infty}|O|^{1/2}}{1-9|O|( \lambda^2+\alpha)},
    \end{align}
    provided that the denominator of the last expression is positive. Note that under the same condition, one can bound the norm
    ${\|\hat\bfDelta\|}_{2,1}$ as follows:
    \begin{align}\label{eqn:4}
    {\|\hat\bfDelta\|}_{2,1}
      \le 3{\|\hat\bfDelta_{O,\bullet}\|}_{2,1}
      \le 3|O|^{1/2}{\|\hat\bfDelta_{O,\bullet}\|}_{2,2}
      \le 3|O|^{1/2}{\|\hat\bfDelta\|}_{2,2}
      \le \frac{52\lambda   {\|\bxi^\top\|}_{2,\infty}|O|}{1-9|O|( \lambda^2+\alpha)}.
    \end{align}
	This completes the proof.	
  \end{proof}

The details of Step 4 are postponed to Subsection~\ref{subs:prob}. Let us just stress here that if
for a $\delta\in(0,1)$  we define the event $\mathcal E$ as the one in which the following inequalities are satisfied:
\begin{align*}
\max_{i\in [n],j\in[p]}|\bfZ^j_{i,\bullet}\bepsilon_{\bullet,j}|&\le \sqrt{2\log(2np/\delta)}\\
\min_{j\in [p]} {\|\bfZ^j \bepsilon_{\bullet,j}\|}_2^2 &\ge n-p+1 - 2\sqrt{(n-p+1)\log(2p/\delta)}\ge n/2\\
\sigma_{\min}(\bfX(\bfOmega^*)^{1/2})&\ge \sqrt{(n-|O|)/4}\\
{\|\bfI_n-\bfZ\|}_{\infty,\infty} &\le \frac{8(1+M_\bfE)^2p+16\log(2n/\delta)}{n-|O|}\\
{\|\bepsilon^\top\|}_{2,\infty} &\le \sqrt{n}+\sqrt{2\log(p/\delta)}\le \sqrt{n}\;(1+2^{-3/2}).
\end{align*}
According to Eq.\ \eqref{eq:ND}, Lemma \ref{lem:6} and Lemma~\ref{lem:3} below, as well as the union bound, we
have $\prob(\mathcal E)\ge 1-3\delta$. Furthermore, combining the above upper bound on $\alpha = {\|\bfI_n-\bfZ\|}_{\infty,\infty}$
with the condition of the theorem, we get that $|O|(\lambda^2+\alpha)\le 1/10$ in $\mathcal E$.
Thus, Proposition~\ref{prop:7} implies the claim of Theorem~\ref{thm:1}.

\subsection{Bounds on estimation error of the precision matrix}\label{app:2}

Let us denote by $\hat\bfD$ and $\bfD^*$ the $p\times p$ diagonal matrices with $\hat\bfD_{jj} = \hat\omega_{jj}$
and $\bfD^*_{jj} = \omega^*_{jj}$, respectively. We know that $\hat\bfOmega = \hat\bfB\hat\bfD$ and
$\bfOmega^* = \bfB^*\bfD^*$.  Hence, an upper bound on the error of estimation of $\bfOmega^*$ can be readily
inferred from bounds on the estimation error of $\bfB^*$ and $\bfD^*$. Indeed,
 \begin{align}
  {\|\hat\bfOmega - \bfOmega^*\|}_{2,2} &\le {\|(\hat\bfB - \bfB^*)\hat\bfD\|}_{2,2}+{\|\bfB^*(\hat\bfD-\bfD^*)\|}_{2,2}\nonumber\\
	&\le {\|\hat\bfB - \bfB^*\|}_{2,2}\max_j \hat\omega_{jj}+\sigma_{\max}(\bfOmega^*){\|\hat\bfD(\bfD^*)^{-1}-\bfI_p\|}_{2,2}.\label{eq:N1}
 \end{align}
To formulate the corresponding result, let us define the condition number $\rho^*\ge 1$ by
$(\rho^*)^2 = \sigma_{\max}(\bfOmega^*)/\sigma_{\min}(\bfOmega^*)$. Throughout this proof, we use $C$ as a generic notation for a universal
constant, whose value may change at each appearance.

\begin{lemma}\label{lem:2A}
It holds that
\begin{align}\label{eq:lem:2A1}
 {\|\hat\bfB - \bfB^*\|}_{2,2}\le \sigma_{\min}(\bfX^{(n)})^{-1}\big(\alpha^{1/2} {\|\hat\bfDelta^\bfTheta\|}_{2,1}+
p^{1/2}\max_j{\|(\bfI_n-\bfZ^j)\bxi_{\bullet,j}\|}_{2}\big).
\end{align}
In addition, if $(|O|p) = o(n/\log n)$, there exists a universal constant $C>0$ such that for sufficiently large values of $n$
the inequality
\begin{align}
{\|\hat\bfB - \bfB^*\|}_{2,2} \le C\rho^* \bigg\{M_{\bfE}\frac{|O|p\log n}{n}+\bigg(\frac{p^2\log n}{n}\bigg)^{1/2}\bigg\}
\end{align}
holds with probability larger than $1-(5/n)$.
\end{lemma}
\begin{proof} To ease notation, throughout this proof we write $\bfOmega$ and $\omega_{jj}$ instead of
$\bfOmega^*$ and $\omega_{jj}^*$, respectively.
One can check that $\bfX^{(n)}(\hat\bfB_{\bullet,j} - \bfB^*_{\bullet,j}) = \bfX^{(n)}_{\bullet,j^c}
(\hat\bfB_{j^c,j} - \bfB^*_{j^c,j})=(\bfI_n-\bfZ^j)(\hat\bfDelta^\bfTheta_{\bullet,j}-\bxi_{\bullet,j})$
for every $j\in[p]$. Therefore, by the triangle inequality, we get  ${\|\bfX^{(n)}(\hat\bfB - \bfB^*)\|}_{2,2} \le
{\|(\bfI_n-\bfZ)\hat\bfDelta^\bfTheta\|}_{2,2}+p^{1/2}\max_j{\|(\bfI_n-\bfZ^j)\bxi_{\bullet,j}\|}_{2}$. We have already used
in the previous section the inequality ${\|(\bfI_n-\bfZ)\hat\bfDelta^\bfTheta\|}_{2,2}\le \alpha^{1/2}
{\|\hat\bfDelta^\bfTheta\|}_{2,1}$. This yields
\begin{align*}
{\|\hat\bfB - \bfB^*\|}_{2,2} \le
\sigma_{\min}(\bfX^{(n)})^{-1}\big({\alpha^{1/2}\|\hat\bfDelta^\bfTheta\|}_{2,1} + p^{1/2}\max_j{\|(\bfI_n-\bfZ^j)\bxi_{\bullet,j}\|}_{2}\big).
\end{align*}
Combining inequality $\sigma_{\min}(\bfX^{(n)})\ge \sigma_{\min}(\bfX^{(n)}\bfOmega^{1/2})\sigma_{\min}(\bfOmega^{-1/2})
=\sigma_{\min}(\bfX^{(n)}\bfOmega^{1/2})\sigma_{\max}(\bfOmega)^{-1/2}$ with the last claim of Lemma~\ref{lem:6} (with $\delta=1/n$),
for $n$ sufficiently large, we get that the inequality $\sigma_{\min}(\bfX^{(n)})\ge C\sigma_{\max}(\bfOmega)^{-1/2}$
holds with probability at least $1-1/n$.  Similarly, using Theorem~\ref{thm:1} with $\delta=1/n$
we check that for $n$ large enough, with probability at least $1-3/n$, we have ${\|\hat\bfDelta^\bfTheta\|}_{2,1}\le
3C_0(\max_j \omega_{jj}^{-1/2})|O|(\frac{p\log n}{n})^{1/2}$. In order to evaluate the term ${\|(\bfI_n-\bfZ^j)\bxi_{\bullet,j}\|}_{2}$,
we note that its square is drawn from the scaled khi-square distribution $(n\omega_{jj})^{-1}\chi^2_{p-1}$.
Therefore, applying the same argument as in Lemma~\ref{lem:3}, we check that with probability at least $1-1/n$,
\begin{align*}
\max_{j\in [p]} {\|(\bfI_n-\bfZ^j)\bxi_{\bullet,j}\|}_{2}\le \max_j (n\omega_{jj})^{-1/2} (\sqrt{p-1}+\sqrt{2\log(pn)})\le
3 \max_j \omega_{jj}^{-1/2} \Big(\frac{p\log n}{n}\Big)^{1/2}.
\end{align*}
In addition, it is clear that $\max_j \omega_{jj}^{-1/2}=
(\min_j \omega_{jj})^{-1/2}\le \sigma_{\min}(\bfOmega)^{-1/2}$. Putting all these bounds together, we obtain the claimed
result.
\end{proof}

\begin{lemma}\label{lem:2B}
If $|O|p = o(n/\log n)$ then there exists a universal constant $C$ such that for $n$ large enough, the inequalities
\begin{align}
\max_j \frac{\hat\omega_{jj}}{\omega^*_{jj}} \le C,\qquad
{\|\hat\bfD(\bfD^*)^{-1}-\bfI_p\|}_{2,2} \le C\Big\{ \rho^*M_{\bfE}\frac{|O|p\log n}{n}+ \Big(\frac{p\log n}{n}\Big)^{1/2}\Big\}
\end{align}
hold with probability larger than $1-4/n$.
\end{lemma}
\begin{proof}
To ease notation, we write $\omega_{jj}$ instead of $\omega^*_{jj}$ and $C_\omega$ for $\max_j \omega_{jj}^{1/2}$.
Let us consider the first term in the right-hand side of the
above inequality. Recall that the diagonal entries $\omega_{jj}$ are estimated by
\begin{align*}
\hat\omega_{jj} = \frac{2n}{\pi \|\bfZ^j(\bfX^{(n)}_{\bullet,j}-\hat\bfTheta_{\bullet,j})\|_1^2}
								= \frac{2n}{\pi\|\hat\bxi_{\bullet,j}\|_1^2}.
\end{align*}
This implies that
\begin{align}
\Big|\Big(\frac{\omega_{jj}}{\hat\omega_{jj}}\Big)^{\frac12} -1\Big| & =
\Big|\Big(\frac{\pi\omega_{jj}}{2n}\Big)^{\frac12}\,{\|\hat\bxi_{\bullet,j}\|}_1 -1\Big|\nonumber\\
& \le \Big(\frac{\pi\omega_{jj}}{2n}\Big)^{\frac12}\,\big|{\|\hat\bxi_{\bullet,j}\|}_1 - {\|\bxi_{\bullet,j}\|}_1\big|+
\Big|\Big(\frac{\pi\omega_{jj}}{2n}\Big)^{\frac12}\,{\|\bxi_{\bullet,j}\|}_1 -1\Big|\nonumber\\
&\le \Big(\frac{\pi\omega_{jj}}{2n}\Big)^{\frac12}\,\big({\|\hat\bxi_{\bullet,j} - \bfZ^j\bxi_{\bullet,j}\|}_1+{\|(\bfI_n-\bfZ^j)\bxi_{\bullet,j}\|}_1\big)+
\Big|\Big(\frac{\pi\omega_{jj}}{2n}\Big)^{\frac12}\,{\|\bxi_{\bullet,j}\|}_1 -1\Big|.\label{eq:c}
\end{align}
The first term above can be bounded using Theorem~\ref{thm:1} since  $\|\hat\bxi_{\bullet,j}-\bfZ^j\bxi_{\bullet,j}\|_1
= {\|\bfZ^j\hat\bfDelta^{\bfTheta}_{\bullet,j}\|}_1$
and
\begin{align}
\|\hat\bxi_{\bullet,j}-\bfZ^j\bxi_{\bullet,j}\|_1
        &\le {\|\hat\bfDelta^{\bfTheta}_{\bullet,j}\|}_1+{\|(\bfI_n-\bfZ^j)\hat\bfDelta^{\bfTheta}_{\bullet,j}\|}_1\nonumber\\
        &\le {\|\hat\bfDelta^{\bfTheta}_{\bullet,j}\|}_1+\sqrt{n}\,{\|(\bfI_n-\bfZ^j)\hat\bfDelta^{\bfTheta}_{\bullet,j}\|}_2\nonumber\\
        &\le {\|\hat\bfDelta^{\bfTheta}_{\bullet,j}\|}_1+\sqrt{n}\,{\|(\bfI_n-\bfZ)\hat\bfDelta^{\bfTheta}_{\bullet,j}\|}_2.
\end{align}
Note that the result of Theorem~\ref{thm:1} applies to this matrix as well.
For the second term of the right-hand side of \eqref{eq:c}, we can use the Cauchy-Schwarz inequality in conjunction with the
fact that $n\omega_{jj}\,\|(\bfI_n-\bfZ^j)\bxi_{\bullet,j}\|_2^2$ is a khi-square random variable with $p-1$ degrees of freedom
degrees of freedom and apply Lemma 1 of \citet{LaurentMassart}.
By the Minkowski inequality, this readily yields that for $n$ large enough, the inequality
\begin{align*}
\|(\bfD^*)^{1/2}\hat\bfD^{-1/2}-\bfI_p\|_{2,2}^2
	&=\bigg\{\sum_{j\in[p]}\Big|\Big(\frac{\omega_{jj}}{\hat\omega_{jj}}\Big)^{1/2} -1\Big|^2\bigg\}\\
	&\le C\Big( \frac{C_\omega^2}n\sum_{j=1}^p {\|\hat\bfDelta^{\bfTheta}_{\bullet,j}\|}_1^2+ C_\omega^2{\|(\bfI_n-\bfZ)\hat\bfDelta^{\bfTheta}\|}_{2,2}^2+ \frac{p\log n}{n}\Big).
\end{align*}
holds with probability at least $1-2/n$. One can show that $\sum_{j=1}^p {\|\hat\bfDelta^{\bfTheta}_{\bullet,j}\|}_1^2\le {\|\hat\bfDelta^{\bfTheta}\|}_{2,1}^2$ and
${\|(\bfI_n-\bfZ)\hat\bfDelta^{\bfTheta}\|}_{2,2}^2\le \alpha {\|\hat\bfDelta^{\bfTheta}\|}_{2,1}^2$ (see the proof of Prop.~\ref{prop:7}). Combining with Theorem~\ref{thm:1} and
Eq.~\eqref{eq:lem6:1}, this yields
\begin{align*}
\|(\bfD^*)^{1/2}\hat\bfD^{-1/2}-\bfI_p\|_{2,2}^2
	&\le C\Big( (\rho^*)^2M_{\bfE}^2\frac{|O|^2p^2\log n}{n^2}+ \frac{p\log n}{n}\Big).
\end{align*}
On the other hand, on the same event, we have
\begin{align*}
{\big\|\hat\bxi_{\bullet,j}\big\|}_1
    &\ge {\big\|\bfZ^j \bxi_{\bullet,j}\big\|}_1 -{\|\bfZ^j\hat\bfDelta^{\bfTheta}_{\bullet,j}\|}_{1}
     \ge {\big\|\bxi_{\bullet,j}\big\|}_1-{\big\|(\bfI_n-\bfZ^j) \bxi_{\bullet,j}\big\|}_1 -\sqrt{n}{\|\hat\bfDelta^{\bfTheta}\|}_{2,2}\\
    &\ge \sqrt{n}\Big(\frac{C}{\omega_{jj}^{1/2}} -{\|\hat\bfDelta^{\bfTheta}\|}_{2,2}\Big).
\end{align*}
Therefore, for $n$ large enough, as we assume that $|O|p = o(n/\log n)$, with probability at least $1-4/n$ we have ${\big\|\hat\bxi_{\bullet,j}\big\|}_2
\ge \frac{Cn^{1/2}}{2\omega_{jj}^{1/2}}$ for all $j\in[p]$ and hence $\max_j \hat\omega_{jj}/\omega_{jj} \le C$.

For the second claim of the lemma, we use the inequalities
\begin{align*}
{\|\hat\bfD(\bfD^*)^{-1}-\bfI_p\|}_{2,2}
		& \le 2\max_j \frac{\hat\omega_{jj}\vee \omega_{jj}}{\omega_{jj}}\; {\big\|(\bfD^*)^{1/2}\hat\bfD^{-1/2}-\bfI_p\big\|}_{2,2}\\
		& \le C \big(\max_{j} \omega_{jj}^{1/2}{\|\hat\bfDelta^\bfTheta\|}_{2,2}+ (p\log n/n)^{1/2}\big)\\
		& \le C\Big\{ \rho^*M_{\bfE}\frac{|O|p\log n}{n}+ \Big(\frac{p\log n}{n}\Big)^{1/2}\Big\}.
\end{align*}
This completes the proof of the lemma.
\end{proof}
The claim of Theorem~\ref{thm:2} readily follows from Lemmas~\ref{lem:2A} and \ref{lem:2B}, in conjunction with
\eqref{eq:N1}.

\begin{remark}
A careful inspection of the above proof shows that the term $\frac{|O|p\log n}{n}$ comes from the use of the inequality
${\|(\bfI_n-\bfZ)\hat\bfDelta^{\bfTheta}\|}_{2,2}\le \alpha^{1/2} {\|\hat\bfDelta^{\bfTheta}\|}_{2,1}$. This simple
inequality is in fact somewhat rough, but (our firm conviction is that) under the assumptions required in this work
the aforementioned rough inequality is sufficient for getting sharp results. A tighter upper bound on the quantity
${\|(\bfI_n-\bfZ)\hat\bfDelta^{\bfTheta}\|}_{2,2}$ can be deduced as follows. First remark that
\begin{align*}
{\|(\bfI_n-\bfZ)\hat\bfDelta^{\bfTheta}\|}_{2,2}\le {\|(\bfY^\top\bfY)^{-1/2}\bfY^\top\hat\bfDelta^{\bfTheta}\|}_{2,2}+
{\|(\bfY^\top\bfY)^{-1/2}\bfE^\top\hat\bfDelta^{\bfTheta}\|}_{2,2}.
\end{align*}
Using the same arguments as in \citep{RWY}, we can establish that for $n$ large enough, with probability close to one, we have
the inequality ${\|(\bfY^\top\bfY)^{-1/2}\bfY^\top\hat\bfDelta^{\bfTheta}\|}_{2,2}\le (p/n)^{1/2}{\|\hat\bfDelta^{\bfTheta}\|}_{2,2}+
(\log n/n)^{1/2}{\|\hat\bfDelta^{\bfTheta}\|}_{2,1}=\frac{|O|(p\log n)^{1/2}}{n}(p^{1/2}+|O|^{1/2})$. On the other hand,
with high probability, (when $p$ is smaller than $n$), the term ${\|(\bfY^\top\bfY)^{-1/2}\bfE^\top\hat\bfDelta^{\bfTheta}\|}_{2,2}$
can be bounded by
\begin{align}\label{sharper}
\sigma_{\min}(\bfY\bfSigma^{-1/2})^{-1}{\|\bfSigma^{-1/2}\bfE^\top\hat\bfDelta^{\bfTheta}\|}_{2,2}\le
\frac{\sigma_{\max}(\bfE\bfSigma^{-1/2})}{\sigma_{\min}(\bfY\bfSigma^{-1/2})}\,
{\|\hat\bfDelta^{\bfTheta}\|}_{2,2}.
\end{align}
If only condition (C2) is assumed, then the inequality ${\sigma_{\max}(\bfE\bfSigma^{-1/2})}\le M_{\bfE}(p|O|)^{1/2}$ holds
and is not improvable (one has equality for the matrix with all the entries equal to $M_{\bfE}$). That is why bound (\ref{sharper})
does not lead to sharper rate under (C2). However, if we consider, for instance, the Huber contamination model then,
under additional mild assumptions on the distribution of the contamination, the term ${\sigma_{\max}(\bfE\bfSigma^{-1/2})}$ will be
of the smaller order $p^{1/2}+|O|^{1/2}$. In such a situation, the foregoing inequalities lead to the minimax rate of estimation
obtained in \citep{ChenGaoRen2015}.
\end{remark}

\subsection{Probabilistic bounds}\label{subs:prob}

This section is devoted to the establishing non-asymptotic bounds on the stochastic terms encountered during the
evaluation of the estimation error.

\begin{lemma} \label{lem:2}
For any $\delta\in (0,1)$, the inequality
 \begin{equation*}
  \max_{i\in[n]} \max_{j\in [p]} \frac{(\bfZ^j_{i,\bullet}\bepsilon_{\bullet,j})^2}{{\|{\displaystyle\bfZ^j} \bepsilon_{\bullet,j}\|}_2^2}\le
	\frac{2\log(2np/\delta)}{n-p+1-2((n-p+1)\log(2p/\delta))^{1/2}}
 \end{equation*}
 holds with probability at least $1-\delta$. Furthermore, if $n\ge 8p+16\log(4/\delta)$ then
 \begin{equation*}
  \max_{i\in[n]} \max_{j\in [p]} \frac{(\bfZ^j_{i,\bullet}\bepsilon_{\bullet,j})^2}{{\|{\displaystyle\bfZ^j} \bepsilon_{\bullet,j}\|}_2^2}\le
	\frac{4\log(2np/\delta)}{n}
 \end{equation*}
 holds with probability at least $1-\delta$.
\end{lemma}

\begin{proof}
 Let us introduce the following random variables
 \begin{align*}
  N_{ij} := \bfZ^j_{i,\bullet}\bepsilon_{\bullet,j} \quad\text{and}\quad
  D_{j} := {\|\bfZ^j \bepsilon_{\bullet,j}\|}_2^2.
 \end{align*}
 The random vector $\bepsilon_{\bullet,j}$ being Gaussian and independent of $\bfX_{\bullet, j^c}$, we infer that
conditionally to $\bfZ^j$, the random variable $N_{ij}$ is drawn from a zero mean Gaussian distribution. Furthermore, its
conditional variance given $\bfZ^j$ equals $\bfZ^j_{i,\bullet}(\bfZ^j_{i,\bullet})^\top=\bfZ^j_{i,i}$ and, therefore is less than or equal to $1$. (Here, we have used
the fact that $\bfZ^j$ is symmetric, idempotent and that all the entries of a projection matrix are in absolute value smaller than or equal to $1$.)
This implies that for any $\delta>0$, it holds that
$\prob\big(\max_{i\in [n],j\in[p]}|N_{ij}|>\sqrt{2\log(2np/\delta)}\big)\leq \delta/2$.

We know that $\bfZ^j$ is an orthogonal projection matrix onto a subspace of dimension $\rank(\bfZ^j)$. We recall that the square of the Euclidean norm of the orthogonal projection in a subspace of dimension $k$ of a standard Gaussian random vector is a $\chi^2$ random variable with $k$ degrees of freedom. It entails that, conditionally to $\bfZ^j$, $D_{j}$ has a $\chi^2$ distribution with $\rank(\bfZ^j)$ degrees of freedom. Therefore, noticing that $\rank(\bfZ^j)\ge n-\rank(\bfX_{\bullet,j^c}) = n-p+1$ almost surely and using a prominent result on tail bounds for the $\chi^2$ distribution (see Lemma 1 of \cite{LaurentMassart}), we get, for every $\delta \in (0,1)$
 \begin{align*}
  \prob\big(\min_{j\in [p]} D_{j} \le n-p+1 - 2\sqrt{(n-p+1)\log(2p/\delta)}\big)
	\le \delta/2.
 \end{align*}
 Thus, on an event of probability at least $1-\delta$, we have
\begin{align}\label{eq:ND}
\max_{\substack{i\in [n]\\j\in[p]}}|N_{ij}|\le \sqrt{2\log(2np/\delta)}\qquad\text{and}
\qquad \min_{j\in [p]} D_{j} \ge n-p+1 - 2\sqrt{(n-p+1)\log(2p/\delta)}.
\end{align}
This readily entails the first claim of the lemma. The second claim follows from the first one. Indeed,
$n\ge 8p+16\log(4/\delta)$ implies that $3p+8\log(4/\delta)\le 0.5n-p$ and, hence,
\begin{align*}
16(n-p+1)\log(2p/\delta)
		&\le \big(0.5(n-p+1)+8\log(2p/\delta)\big)^2\\
		&\le \big(0.5n-p+1+0.5p+8\log(p/2)+8\log(4/\delta)\big)^2\\
		&\le \big(0.5n-p+1+3p+8\log(4/\delta)\big)^2\\
		&\le \big(n-2p+1\big)^2.
\end{align*}
This yields $n-p+1-2((n-p+1)\log(2p/\delta))^{1/2}\ge n/2$.
\end{proof}

The element-wise $\ell_\infty$-norm of the orthogonal projection matrix $\bfI_n-\bfZ$ also appears in the upper bounds of the estimation error.
Lemma \ref{lem:6} below provides a sharp tail bound for this norm. Before showing this result, let us provide a useful technical lemma that
relies essentially on a lower bound for the smallest singular value of a Gaussian matrix.

\begin{lemma}\label{lem:7}
If $\bfX$ is an $n \times p$ random matrix satisfying conditions {\bf (C1)} and {\bf (C2)} with $\bfSigma^*=\bfI_p$, then
for every $\delta\in(0,1)$, with probability at least $1-\delta$, it holds that
 \begin{align*}
  \sigma_{\min}(\bfX)\ge \sqrt{n-|O|}-\sqrt{p}-\sqrt{2\log(2/\delta)}.
 \end{align*}
\end{lemma}

\begin{proof}
 To begin, we note that the matrix $\bfX^\top \bfX$ can be split into two parts, by summing the terms derived from inliers ($I\subset[n]$) on one hand and those derived from outliers ($O\subset[n]$) on the other hand,
 \begin{align*}
  \bfX^\top\bfX = \sum_{i\in[n]} \bfX_{i,\bullet}^\top\bfX_{i,\bullet} = \sum_{i\in I} \bfX_{i,\bullet}^\top\bfX_{i,\bullet} + \sum_{i\in O} \bfX_{i,\bullet}^\top\bfX_{i,\bullet} = \bfX_{I,\bullet}^\top\bfX_{I,\bullet} + \bfX_{O,\bullet}^\top\bfX_{O,\bullet}.
 \end{align*}
 As the matrix $\bfX_{O,\bullet}^\top\bfX_{O,\bullet}$ is always nonnegative definite and
 $\bfX^\top\bfX=\bfX_{O,\bullet}^\top\bfX_{O,\bullet}+\bfX_{I,\bullet}^\top\bfX_{I,\bullet}$, we infer that
 $\sigma_{\min}(\bfX^\top\bfX) \ge \sigma_{\min}(\bfX_{I,\bullet}^\top\bfX_{I,\bullet})$.
 We can therefore deduce that
 \begin{align*}
  \sigma_{\min}(\bfX) = \sigma_{\min}(\bfX^\top\bfX)^{1/2} \ge \sigma_{\min}(\bfX_{I,\bullet}^\top\bfX_{I,\bullet})^{1/2}
	= \sigma_{\min}(\bfX_{I,\bullet}).
 \end{align*}
 Given that $\bfX_{I,\bullet}$ is a matrix whose rows are independent Gaussian vectors with zero-mean and identity covariance, as shown in \cite[Corollary 5.35]{Vershynin2012}, for every $t\ge0$, it holds that
 \begin{align*}
  \sigma_{\min}(\bfX_{I,\bullet}) \ge \sqrt{|I|}-\sqrt{p}-t.
 \end{align*}
 with probability at least $1-2\e^{-t^2/2}$. Taking $ t = \sqrt{2\log(2/\delta)}$, the claim of the lemma follows.
\end{proof}

\begin{lemma}\label{lem:6}
If $\bfX= \bfY+\bfE^*$ is an $n \times p$ random matrix with $\bfY$ and $\bfE^*$ satisfying assumptions {\bf (C1)} and {\bf (C2)} with $\bmu^*=0$, then
 for any $\delta\in(0,1)$, the inequality
 \begin{equation*}
  {\|\bfI_n-\bfZ\|}_{\infty,\infty} \le \bigg(\frac{(1+M_\bfE)\sqrt{p}+\sqrt{2\log(2n/\delta)}}{\sqrt{n-|O|}-\sqrt{p}-\sqrt{2\log(4/\delta)}}\bigg)^2,
 \end{equation*}
holds with probability at least $1-\delta$. Furthermore, if $n\ge |O|+8p+16\log(4/\delta)$, then with probability at least $1-\delta$,
 \begin{equation}\label{eq:lem6:1}
  {\|\bfI_n-\bfZ\|}_{\infty,\infty} \le \frac{8(1+M_\bfE)^2p+16\log(2n/\delta)}{n-|O|}.
 \end{equation}
 and $\sigma_{\min}(\bfX(\bfOmega^*)^{1/2})\ge \sqrt{(n-|O|)/4}$.
\end{lemma}

\begin{proof}
 We denote by $\{\bfe_i\}_{i\in[n]}\subset\R^n$ the vectors of the canonical basis. All the components of the vector $\bfe_i\in\R^n$ are equal to zero with the exception of the $i$-th entry which is equal to one.
 With this notation, and using the fact that all the off-diagonal entries of a symmetric positive semi-definite matrix are dominated by the largest diagonal entry, we have
 \begin{align*}
  {\|\bfI_n-\bfZ\|}_{\infty,\infty} = \max_{i\in[n]} \bfe_i^\top(\bfI_n-\bfZ)\bfe_{i}.
 \end{align*}
 We also denote $\bfX(\bfSigma^*)^{-1/2}$ by $\tilde\bfX$ and, similarly, $\bfY(\bfSigma^*)^{-1/2}$ by $\tilde\bfY$. It follows that for any $i\in[n]$
 \begin{align*}
  \bfe_i^\top(\bfI_n-\bfZ)\bfe_{i}= \bfe_i^\top \tilde\bfX \big(\tilde\bfX^\top \tilde\bfX \big)^\dag \tilde\bfX^\top\bfe_{i}\le {\|\tilde\bfX_{i,\bullet}\|}_2^2 \sigma_{\max}\big((\tilde\bfX^\top \tilde\bfX)^\dag\big)	.
 \end{align*}
 where the last inequality is a direct consequence of the fact that the spectral norm is the matrix norm induced by the Euclidean norm.
 We may now bound each term of the right side of the previous inequality.
 First, by assumption, it holds that
 \begin{align*}
  {\|\tilde\bfX_{i,\bullet}\|}_2 =  {\|\tilde\bfY_{i,\bullet}+\bfE^*_{i,\bullet}(\bfSigma^*)^{-1/2}\|}_2 \le {\|\tilde\bfY_{i,\bullet}\|}_2+{\|\bfE^*_{i,\bullet}(\bfSigma^*)^{-1/2}\|}_2\le
	{\|\tilde\bfY_{i,\bullet}\|}_2+M_{\bfE}\sqrt{p}.
 \end{align*}
 As $\tilde\bfY_{i,\bullet} \sim \mathcal N_p(0,\bfI_p)$, the random variable
${\|\tilde\bfY_{i,\bullet}\|}_2$ has a $\chi^2$ distribution with $p$ degrees of freedom.
 Applying \citep[Lemma 1]{LaurentMassart} and combining it with the union bound, for any $\delta\in(0,1)$, we get that
 \begin{align*}
 \max_{i\in [n]} {\|\tilde\bfY_{i,\bullet}\|}_2 \le \sqrt{p}+\sqrt{2\log(2n/\delta)},
 \end{align*}
 with a probability at least $1-\delta/2$.
 We complete the proof by bounding $\sigma_{\max}(( \tilde\bfX^\top \tilde\bfX)^\dag)=\sigma_{\min}(\tilde\bfX)^{-2}$. By Lemma \ref{lem:7}, for every $\delta\in(0,1)$, it holds that
 \begin{align*}
\sigma_{\max}(( \tilde\bfX^\top \tilde\bfX)^\dag)\le (\sqrt{|I|}-\sqrt{p}-\sqrt{2\log(4/\delta)})^{-2}
 \end{align*}
 with probability at least $1-\delta/2$. By bringing together what was written above, with a probability at least $1-\delta$, we have
 \begin{align*}
  \max_{i\in[n]}\bfe_i^\top(\bfI_n-\bfZ)\bfe_{i} \le \bigg(\frac{(1+M_\bfE)\sqrt{p}+\sqrt{2\log(2n/\delta)}}
  {\sqrt{|I|}-\sqrt{p}-\sqrt{2\log(4/\delta)}}\bigg)^2.
 \end{align*}
 This yields the first claim of the lemma. To derive the second claim from the first one, it suffices to upper bound the numerator using the inequality
 $(a+b)^2\le 2a^2+2b^2$ and to lower bound the denominator by using that $\sqrt{p}+\sqrt{2\log(4/\delta)}\le \sqrt{2p+4\log(4/\delta)}\le \frac12\sqrt{n-|O|}$.
\end{proof}

\begin{lemma} \label{lem:3}
 For any $\delta \in (0,1)$, the following inequality
 \begin{align}
  {\|\bepsilon^\top\|}_{2,\infty} \le \sqrt{n}+\sqrt{2\log(p/\delta)},
 \end{align}
 holds with probability at least $1 - \delta$.
\end{lemma}

\begin{proof}
 We recall that
 ${\|\bepsilon^\top\|}_{2,\infty} = \max_{j\in[p]} {\|\bepsilon_{\bullet,j}\|}_2$.
 As we have already mentioned just after equation \eqref{main:eq}, the vector $\bepsilon_{\bullet,j}$ is drawn from the
 Gaussian $\mathcal N_n(0,\bfI_n)$ distribution. Therefore, ${\|\bepsilon_{\bullet,j}\|}^2_2$ is a $\chi^2$ random variable
 with $n$ degrees of freedom.  Thus, using \citep[Lemma 1]{LaurentMassart} in combination with the union bound, it holds that
 \begin{align*}
  {\|\bepsilon^\top\|}_{2,\infty}^2 \le n+2\sqrt{n\log(p/\delta)}+2\log(p/\delta)\le (\sqrt{n}+\sqrt{2\log(p/\delta)})^2,
 \end{align*}
 with probability at least $1-\delta$.
\end{proof}

%%%%%%%%%%%%%%%%%%%%%%%%%%%%%%%%%%%%%%%%%%%%%%%%%%%%%%%%%%%%%%%%%%%%%%%%%%%%%%%
\section{Numerical experiments}\label{sec:5}

In this section, we report the results of some numerical experiments performed on synthetic data. The main goal of this part
is to demonstrate the potential of the method based on Eq.~\eqref{step:1bis} and \eqref{step:2bis}. To this end, we have
considered several scenarios and in each of them compared our method with several other competitors. In order to provide a fair
comparison independent of the delicate question of choosing the tuning parameter, the results of all the methods are reported
for the oracle values of the tuning parameters chosen from a grid by minimizing the distance to the true precision matrix.
We have used the coordinate descent algorithm for solving the convex optimization problem of Eq.~\eqref{step:1bis}.

\subsection{Structures of the precision matrix}

Let us first describe the precision matrices used in our experiments. It is worthwhile to underline here that all the precision
matrices are normalized in such a way that all the diagonal entries of the corresponding covariance matrix
$\bfSigma^* = (\bfOmega^*)^{-1}$ are equal to one. To this end, we first define a $p\times p$ positive semidefinite matrix $\bfA$
and then set $\bfOmega^* = (\diag(\bfA^{-1}))^{\frac1{2}} \bfA (\diag(\bfA^{-1}))^{\frac1{2}}$. The matrices $\bfA$ used in the three models
for which the experiments are carried out are defined as follows.

\begin{description}\itemsep=4pt
\item[{\bf Model 1:}] $\bfA$ is a Toeplitz matrix with the entries $\bfA_{ij} = 0.6^{|i-j|}$ for any $i,j \in [p]$.

 \item[{\bf Model 2:}] We start by defining a $p\times p$ pentadiagonal matrix with the entries
 \begin{equation*}
 \bar\bfA_{ij} = \left\lbrace \begin{array}{cl}
                                                       1 &, \text{ for } |i-j| = 0, \\
                                                       -1/3 &, \text{ for } |i-j| = 1, \\
                                                       -1/10 &, \text{ for } |i-j| = 2, \\
                                                       0 &, \text{ otherwise}. \\
                                                      \end{array}
 \right.
 \end{equation*}
Then, we denote by $\bfA$ the matrix with the entries $\bfA_{ij} = (\bar\bfA^{-1})_{ij} \1(|i-j|\le 2)$. One can check that the matrix $\bfA$
defined in such a way is positive semidefinite.

\item[{\bf Model 3:}] We set $\bfA_{ij} = 0$ for all the off-diagonal entries that are neither on the first row nor on the first
column of $\bfA$. The diagonal entries of $\bfA$ are
\begin{equation*}
\bfA_{11}  = p,\qquad \bfA_{ii} = 2,\quad \text{for any}\quad i \in \{2,\ldots,p\},
\end{equation*}
whereas the off-diagonal entries located either on the first row or on the first column are $\bfA_{1i} = \bfA_{i1} = \sqrt{2}$ for $i \in \{2,\ldots,p\}$.

\item[{\bf Model 4:}] The diagonal entries of $\bfA$ are all equal to 1. Besides, we set $\bfA_{ij} = 0.5$ for any $i\ne j$.
\end{description}

\subsection{Contamination scheme and measure of quality}

The positions of outliers were chosen by a simple random sampling without replacement. The proportion of outliers,
$\epsilon = |O|/n$, used in our experiments varies between $5\%$ and $30\%$.
The entries of the rows of $\bfX$ corresponding to outliers were drawn randomly from a standard
Gaussian distribution and independently of one another. The rows of $\bfX$ corresponding to inliers
are drawn from  a zero mean Gaussian distribution with the precision matrix specified by one of the
foregoing models. Note that the magnitude of the individual entries of outliers are similar to those
of the inliers, which makes the outliers  particularly hard to detect.

We measure the distance between the true precision matrix of a multivariate normal distribution
and its estimator using the distance induced by the Frobenius norm. Recall that our method does
not guarantee the positive definiteness of the estimate of the precision matrix.
When the estimate is not positive definite, one can always get a valid precision matrix from $\hat\bfOmega$.
A number of methods have been proposed in the literature for adjusting a matrix such that it is positive definite.
In practice, replacing $\hat\bfOmega$ by the positive definite matrix obtained by the approach of
\citet{Higham2002}, seems to be a good choice as it does not significantly affect the norm-induced
distance between the true precision matrix and its estimate.

\subsection{Precision matrix estimators}

We have compared our method to four other estimators of the precision matrix.
The first and the most naive estimator, referred to as the MLE,  consists in computing the (pseudo-)inverse
of the empirical covariance matrix.

The second estimator is the inverse of a robust covariance estimate
computed by the minimum covariance determinant (MCD) method introduced in \citep{Rousseeuw1984}.
We have used a shrinkage coefficient coming from the improvement of the Ledoit-Wolf shrinkage, developed by
\citet{ChenWieselEldarHero2010} for multivariate Gaussian distributions. We therefore refined the MCD estimator
using the covariance Oracle Shrinkage Approximating (OAS) estimator. In the following, we refer to it as SMCD.
We also did experiments estimating the covariance matrix by the minimum volume ellipsoid (MVE) estimator \citep{Rousseeuw1985} and
by the scaled Kendall's tau estimator \citep{ChenGaoRen2015}. The results obtained for the latter estimators
are not reported as they showed no improvement over the SMCD.

\begin{figure}
 \includegraphics[width = \textwidth]{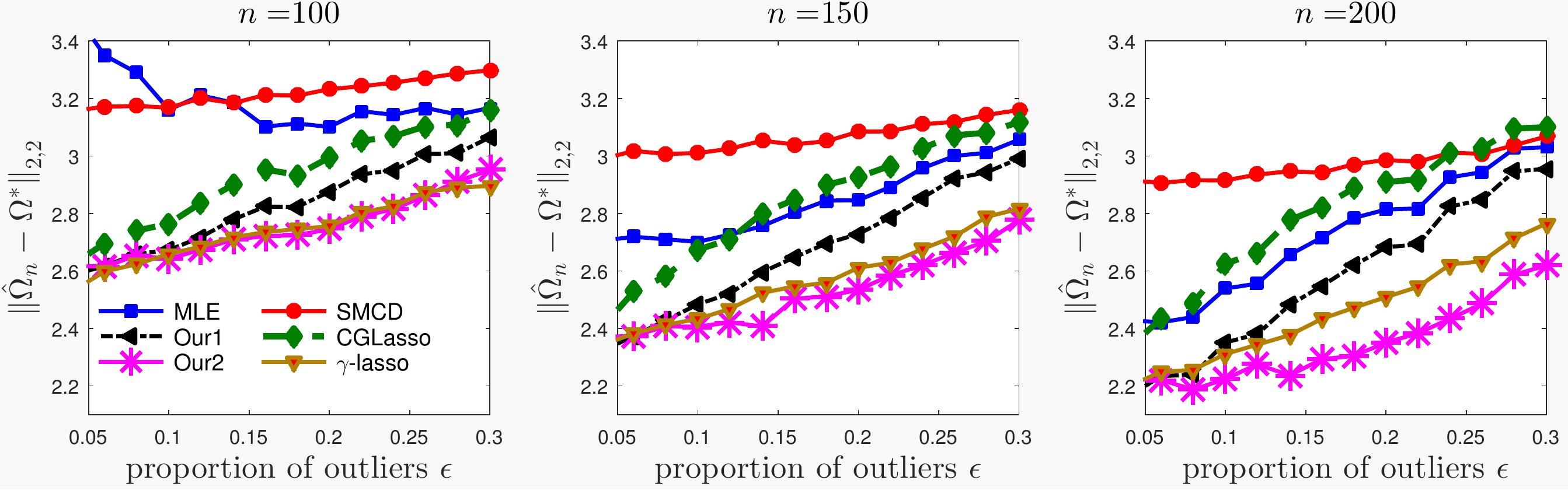}
 \caption{
 The average error (measured in Frobenius norm) of estimating $\bfOmega^*$ in Model 1 for $p=30$, when $\epsilon$ is
 between $5\%$ and $30\%$. Each point is the average of 50 replications.
 }
 \label{fig:p.1}
\end{figure}

The third estimator of the precision matrix is obtained by solving an optimization problem whose cost
function depends on a robust estimate of the covariance matrix. Two versions of this approach are particularly
interesting: the maximum log-likelihood with $\ell_1$-penalization known as graphical lasso
\citep{dAsprBanerjeeEG,BanerjeeGAspr,FriedmanHT} and the constrained $\ell_1$-minimization for inverse matrix
estimation (CLIME) of \citet{CaiLiuLuo}. Robust versions of these estimators have been proposed by \citet{OllererCroux2015}
and \citet{TarrMullerWeber2015} and further investigated by \cite{LohTan2015}. In this approach, robust estimates of the
covariance matrix are plugged-in the graphical lasso or CLIME estimators. In our experiments, the quality of these two
versions were comparable. Therefore, we report only the results for the version based on the graphical lasso.
In \citep{OllererCroux2015}, the authors proposed an enhancement that simplifies the estimator and reduces the
computational cost, by estimating aside the variances and the correlations. Following their work, we chose to
estimate the correlations by the robust Gaussian rank correlation \citep{Boudt2012} and adopted their implementation
choices. In particular, as a robust measure of scale, we used the $Q_n$ estimator of \citet{RousseeuwCroux1993} that
is an alternative to the median absolute deviation (MAD). To sum up, we implemented the correlation based precision
matrix estimator obtained by plugging-in the covariance matrix estimate based on pairwise correlations in the
graphical lasso (hereinafter referred to as CGLASSO).

The fourth estimator used in our experiments is the $\gamma$-LASSO proposed by \citet{HiroseFujisawa2015}.
The crux of the method is the replacement of the penalized negative log-likelihood function by the penalized
negative $\gamma$-likelihood function \citep{Fujisawa:2008:RPE:1434999.1435056,Cichocki2010}. We used the
{\tt R} package  {\tt rsggm} developed by the \citet{HiroseFujisawa2015}.

Finally we considered two version of our approach, referred to as Our1 and Our2. The first version merely
provided by \eqref{step:1bis} and \eqref{step:2bis}, while the second version consists in re-estimating the
precision matrix by the maximum likelihood after removing the observations classified as outliers.

\begin{figure}
 \includegraphics[width = \textwidth]{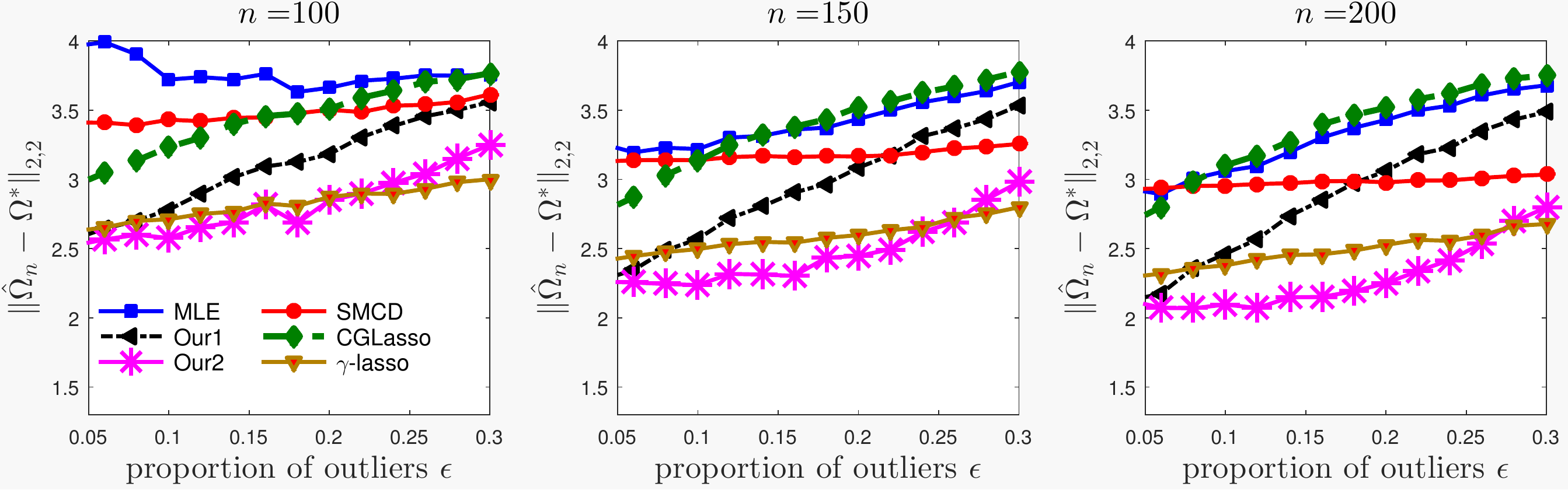}
 \caption{
 The average error (measured in Frobenius norm) of estimating $\bfOmega^*$ in Model 2 for $p=30$, when $\epsilon$ is
 between $5\%$ and $30\%$. Each point is the average of 50 replications.
 }
 \label{fig:p.2}
\end{figure}

\subsection{Results} \label{sec:4:results}

The results of our experiments are depicted in Figures \ref{fig:p.1}-\ref{fig:p.4}. In all the experiments, the
the dimension $p$ is equal to 30 and the contamination rate, denoted by $\epsilon$, is between $5\%$ and $30\%$.
The results show that our procedure is competitive with the state-of-the-art robust estimators of the precision
matrix, even when the proportion of outliers is high. The results for dimensions $p=10,50,100$ were very similar
and therefore are not included in the manuscript.

\begin{figure}
  \includegraphics[width = \textwidth]{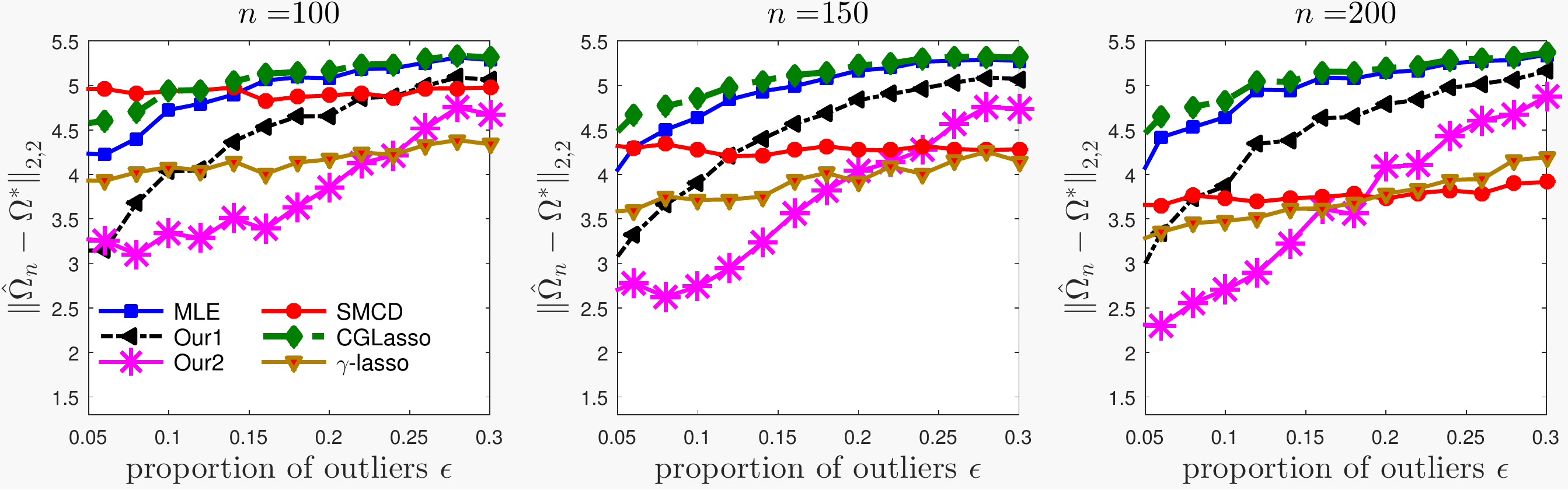}
 \caption{
 The average error (measured in Frobenius norm) of estimating $\bfOmega^*$ in Model 3 for $p=30$, when $\epsilon$ is
 between $5\%$ and $30\%$. Each point is the average of 50 replications.
 }
 \label{fig:p.3}
\end{figure}

One may observe that the step of re-estimation of the precision matrix after the removal of the observations classified
as outliers reduces the error of estimation in all the considered situations. We would also like to mention that the
$\gamma$-lasso, which has a highly competitive statistical accuracy is defined as the minimizer of a nonconvex cost
function. Furthermore, there is no theoretical guarantee ensuring the convergence of the algorithm or controlling
its statistical error.

\begin{figure}
  \includegraphics[width = \textwidth]{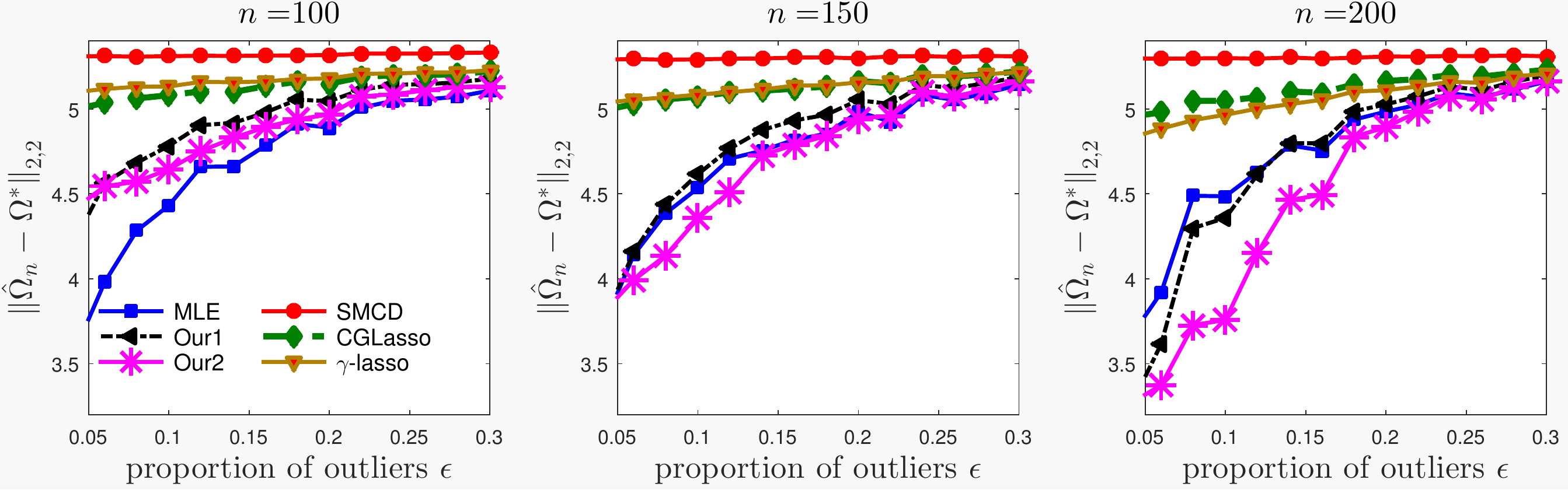}
 \caption{
 The average error (measured in Frobenius norm) of estimating $\bfOmega^*$ in Model 4 for $p=30$, when $\epsilon$ is
 between $5\%$ and $30\%$. Each point is the average of 50 replications.
 }
 \label{fig:p.4}
\end{figure}

\section*{Acknowledgments} The work of the second author was partially supported by the grant Investissements d'Avenir (ANR-
11-IDEX-0003/Labex Ecodec/ANR-11-LABX-0047) and the chair LCL-GENES.

\bibliography{tmpBibliography1}

\begin{thebibliography}{36}
\providecommand{\natexlab}[1]{#1}
\providecommand{\url}[1]{\texttt{#1}}
\expandafter\ifx\csname urlstyle\endcsname\relax
  \providecommand{\doi}[1]{doi: #1}\else
  \providecommand{\doi}{doi: \begingroup \urlstyle{rm}\Url}\fi

\bibitem[Banerjee et~al.(2008)Banerjee, {El Ghaoui}, and
  {d'Aspremont}]{BanerjeeGAspr}
O.~Banerjee, L.~{El Ghaoui}, and A.~{d'Aspremont}.
\newblock {Model selection through sparse maximum likelihood estimation for
  multivariate Gaussian or binary data}.
\newblock \emph{J. Mach. Learn. Res.}, 9:\penalty0 485--516, June 2008.

\bibitem[Belloni et~al.(2011)Belloni, Chernozhukov, and Wang]{BCW}
A.~Belloni, V.~Chernozhukov, and L.~Wang.
\newblock Square-root {L}asso: pivotal recovery of sparse signals via conic
  programming.
\newblock \emph{Biometrika}, 98\penalty0 (4):\penalty0 791--806, December 2011.

\bibitem[Bickel et~al.(2009)Bickel, Ritov, and Tsybakov]{BRT}
P.~J. Bickel, Y.~Ritov, and A.~B. Tsybakov.
\newblock Simultaneous analysis of {L}asso and {D}antzig selector.
\newblock \emph{Ann. Statist.}, 37\penalty0 (4):\penalty0 1705--1732, August
  2009.

\bibitem[Boudt et~al.(2012)Boudt, Cornelissen, and Croux]{Boudt2012}
Kris Boudt, Jonathan Cornelissen, and Christophe Croux.
\newblock The gaussian rank correlation estimator: robustness properties.
\newblock \emph{Statistics and Computing}, 22\penalty0 (2):\penalty0 471--483,
  2012.

\bibitem[Cai et~al.(2011)Cai, Liu, and Luo]{CaiLiuLuo}
T.~Cai, W.~Liu, and X.~Luo.
\newblock {A Constrained L1 Minimization Approach to Sparse Precision Matrix
  Estimation}.
\newblock \emph{Journal of the American Statistical Association}, 106:\penalty0
  594--607, February 2011.

\bibitem[Cand{\`e}s and Randall(2008)]{Candes08}
Emmanuel~J. Cand{\`e}s and Paige~A. Randall.
\newblock Highly robust error correction by convex programming.
\newblock \emph{IEEE Trans. Inform. Theory}, 54\penalty0 (7):\penalty0
  2829--2840, 2008.
\newblock ISSN 0018-9448.
\newblock \doi{10.1109/TIT.2008.924688}.
\newblock URL \url{http://dx.doi.org/10.1109/TIT.2008.924688}.

\bibitem[{Chen} et~al.(2015){Chen}, {Gao}, and {Ren}]{ChenGaoRen2015}
M.~{Chen}, C.~{Gao}, and Z.~{Ren}.
\newblock {Robust Covariance Matrix Estimation via Matrix Depth}.
\newblock \emph{ArXiv e-prints}, June 2015.

\bibitem[Chen et~al.(2015)Chen, Gao, and Ren]{Gao2015}
M.~Chen, C.~Gao, and Z.~Ren.
\newblock {Robust Covariance Matrix Estimation via Matrix Depth}.
\newblock \emph{arXiv:1506.00691}, 2015.

\bibitem[Chen et~al.(2010)Chen, Wiesel, Eldar, and
  Hero]{ChenWieselEldarHero2010}
Yilun Chen, A.~Wiesel, Y.C. Eldar, and A.O. Hero.
\newblock Shrinkage algorithms for mmse covariance estimation.
\newblock \emph{Signal Processing, IEEE Transactions on}, 58\penalty0
  (10):\penalty0 5016--5029, Oct 2010.

\bibitem[Cichocki and Amari(2010)]{Cichocki2010}
A.~Cichocki and S.~Amari.
\newblock Families of alpha- beta- and gamma- divergences: Flexible and robust
  measures of similarities.
\newblock \emph{Entropy}, 12\penalty0 (6):\penalty0 1532, 2010.
\newblock URL \url{http://www.mdpi.com/1099-4300/12/6/1532}.

\bibitem[Dalalyan and Keriven(2012)]{DalalyanK}
Arnak Dalalyan and Renaud Keriven.
\newblock Robust estimation for an inverse problem arising in multiview
  geometry.
\newblock \emph{J. Math. Imaging Vision}, 43\penalty0 (1):\penalty0 10--23,
  2012.

\bibitem[Dalalyan and Chen(2012)]{DalalyanC12}
Arnak~S. Dalalyan and Yin Chen.
\newblock Fused sparsity and robust estimation for linear models with unknown
  variance.
\newblock In \emph{Advances in Neural Information Processing Systems 25}, pages
  1268--1276, 2012.

\bibitem[{d'Aspremont} et~al.(2008){d'Aspremont}, Banerjee, and {El
  Ghaoui}]{dAsprBanerjeeEG}
A.~{d'Aspremont}, O.~Banerjee, and L.~{El Ghaoui}.
\newblock First-order methods for sparse covariance selection.
\newblock \emph{SIAM. J. Matrix Anal. {\&} Appl.}, 30\penalty0 (1):\penalty0
  56--66, February 2008.

\bibitem[Friedman et~al.(2008)Friedman, Hastie, and Tibshirani]{FriedmanHT}
J.~Friedman, T.~Hastie, and R.~Tibshirani.
\newblock {Sparse inverse covariance estimation with the graphical lasso}.
\newblock \emph{Biostatistics}, 9\penalty0 (3):\penalty0 432--441, July 2008.

\bibitem[Fujisawa and Eguchi(2008)]{Fujisawa:2008:RPE:1434999.1435056}
H.~Fujisawa and S.~Eguchi.
\newblock Robust parameter estimation with a small bias against heavy
  contamination.
\newblock \emph{J. Multivar. Anal.}, 99\penalty0 (9):\penalty0 2053--2081,
  October 2008.
\newblock ISSN 0047-259X.
\newblock URL \url{http://dx.doi.org/10.1016/j.jmva.2008.02.004}.

\bibitem[Hampel et~al.(1986)Hampel, Ronchetti, Rousseeuw, and Stahel]{Hampel}
Frank~R. Hampel, Elvezio~M. Ronchetti, Peter~J. Rousseeuw, and Werner~A.
  Stahel.
\newblock \emph{Robust statistics}.
\newblock Wiley Series in Probability and Mathematical Statistics: Probability
  and Mathematical Statistics. John Wiley \& Sons, Inc., New York, 1986.
\newblock The approach based on influence functions.

\bibitem[Higham(2002)]{Higham2002}
N.~J. Higham.
\newblock Computing the nearest correlation matrix—a problem from finance.
\newblock \emph{IMA journal of Numerical Analysis}, 22\penalty0 (3):\penalty0
  329--343, 2002.

\bibitem[{Hirose} and {Fujisawa}(2015)]{HiroseFujisawa2015}
K.~{Hirose} and H.~{Fujisawa}.
\newblock {Robust sparse Gaussian graphical modeling}.
\newblock \emph{ArXiv e-prints}, August 2015.
\newblock URL \url{http://arxiv.org/abs/1508.05571}.

\bibitem[Huber and Ronchetti(2009)]{Huber09}
Peter~J. Huber and Elvezio~M. Ronchetti.
\newblock \emph{Robust statistics}.
\newblock Wiley Series in Probability and Statistics. John Wiley \& Sons, Inc.,
  Hoboken, NJ, second edition, 2009.

\bibitem[{Klopp} and {Tsybakov}(2015)]{KloppTsyb15}
O.~{Klopp} and A.~B. {Tsybakov}.
\newblock {Estimation of matrices with row sparsity}.
\newblock \emph{ArXiv e-prints}, September 2015.

\bibitem[Klopp et~al.(2014)Klopp, Lounici, and Tsybakov]{Klopp15}
O.~Klopp, K.~Lounici, and A.~B. Tsybakov.
\newblock {Robust Matrix Completion}.
\newblock \emph{ArXiv e-prints}, December 2014.

\bibitem[Laurent and Massart(2000)]{LaurentMassart}
B.~Laurent and P.~Massart.
\newblock Adaptive estimation of a quadratic functional by model selection.
\newblock \emph{Ann. Statist.}, 28\penalty0 (5):\penalty0 1302--1338, October
  2000.

\bibitem[{Loh} and {Tan}(2015)]{LohTan2015}
P.-L. {Loh} and X.~L. {Tan}.
\newblock {High-dimensional robust precision matrix estimation: Cellwise
  corruption under $\epsilon$-contamination}.
\newblock \emph{ArXiv e-prints}, September 2015.

\bibitem[Lounici et~al.(2011)Lounici, Pontil, van~de Geer, and
  Tsybakov]{lounici2011}
Karim Lounici, Massimiliano Pontil, Sara van~de Geer, and Alexandre~B.
  Tsybakov.
\newblock Oracle inequalities and optimal inference under group sparsity.
\newblock \emph{Ann. Statist.}, 39\penalty0 (4):\penalty0 2164--2204, 08 2011.

\bibitem[Maronna et~al.(2006)Maronna, Martin, and
  Yohai]{MaronnaMartinYohai2006}
Ricardo~A. Maronna, Douglas~R. Martin, and Victor~J. Yohai.
\newblock \emph{{Robust Statistics: Theory and Methods}}.
\newblock John Wiley and Sons, New York, 2006.

\bibitem[Nguyen and Tran(2013)]{Nguyen}
Nam~H. Nguyen and Trac~D. Tran.
\newblock Robust {L}asso with missing and grossly corrupted observations.
\newblock \emph{IEEE Trans. Inform. Theory}, 59\penalty0 (4):\penalty0
  2036--2058, 2013.

\bibitem[{{\"O}llerer} and {Croux}(2015)]{OllererCroux2015}
V.~{{\"O}llerer} and C.~{Croux}.
\newblock {Robust high-dimensional precision matrix estimation}.
\newblock \emph{ArXiv e-prints}, January 2015.

\bibitem[Raskutti et~al.(2010)Raskutti, Wainwright, and Yu]{RWY}
G.~Raskutti, M.~J. Wainwright, and B.~Yu.
\newblock {Restricted eigenvalue properties for correlated Gaussian designs}.
\newblock \emph{J. Mach. Learn. Res.}, 11:\penalty0 2241--2259, August 2010.

\bibitem[Rousseeuw(1984)]{Rousseeuw1984}
Peter~J Rousseeuw.
\newblock Least median of squares regression.
\newblock \emph{Journal of the American statistical association}, 79\penalty0
  (388):\penalty0 871--880, 1984.

\bibitem[Rousseeuw(1985)]{Rousseeuw1985}
Peter~J Rousseeuw.
\newblock Multivariate estimation with high breakdown point.
\newblock \emph{Mathematical Statistics and Applications}, 8:\penalty0
  283--297, 1985.

\bibitem[Rousseeuw and Croux(1993)]{RousseeuwCroux1993}
Peter~J Rousseeuw and Christophe Croux.
\newblock Alternatives to the median absolute deviation.
\newblock \emph{Journal of the American Statistical association}, 88\penalty0
  (424):\penalty0 1273--1283, 1993.

\bibitem[Sun and Zhang(2012)]{SunZhang12}
T.~Sun and C-H. Zhang.
\newblock Scaled sparse linear regression.
\newblock \emph{Biometrika}, 99\penalty0 (4):\penalty0 879--898, September
  2012.

\bibitem[Sun and Zhang(2013)]{SunZhang13}
T.~Sun and C-H. Zhang.
\newblock Sparse matrix inversion with scaled {L}asso.
\newblock \emph{J. Mach. Learn. Res.}, 14:\penalty0 3385--3418, November 2013.

\bibitem[{Tarr} et~al.(2015){Tarr}, {M{\"u}ller}, and
  {Weber}]{TarrMullerWeber2015}
G.~{Tarr}, S.~{M{\"u}ller}, and N.~C. {Weber}.
\newblock {Robust estimation of precision matrices under cellwise
  contamination}.
\newblock \emph{ArXiv e-prints}, January 2015.

\bibitem[Vershynin(2012)]{Vershynin2012}
R.~Vershynin.
\newblock Introduction to the non-asymptotic analysis of random matrices.
\newblock In \emph{Compressed sensing}, pages 210--268. Cambridge Univ. Press,
  Cambridge, 2012.

\bibitem[Wang and Lin(2014)]{WangLin2014}
J.~Wang and S.~Lin.
\newblock Robust inverse covariance estimation under noisy measurements.
\newblock In Tony Jebara and Eric~P. Xing, editors, \emph{Proceedings of the
  31st International Conference on Machine Learning (ICML-14)}, pages 928--936.
  JMLR Workshop and Conference Proceedings, 2014.
\newblock URL \url{http://jmlr.org/proceedings/papers/v32/wangf14.pdf}.

\end{thebibliography}

\newpage

\section*{Supplementary Material}

\subsection*{Proofs in high dimension}

In this section, we provide the proof of the risk bound in the high dimensional case, when the estimator is
obtained by solving the optimization problem in \eqref{step:1bis}.
We define $\calO = O \times[p]$ and, by a slight abuse of notation, $\calO^c = O^c \times[p]$.
We denote by $\bxi_{\calO}$, resp. $\bxi_{\calO^c}$, the matrix obtained by zeroing all the rows $\bxi_{i,\bullet}$ such that $i \in O$, resp. $i \in O^c$.
We set $\bar\bfTheta^* = \bfTheta^* + \bxi_{\calO}$
and $\bar\bxi = \bxi_{\calO^c}$.
We further define $\hat\bfDelta^\bfB = \hat\bfB - \bfB^*$,
$\hat\bfDelta^\bfTheta = \hat\bfTheta-\bar\bfTheta^*$, $\hat\bfDelta = \begin{bmatrix}\hat\bfDelta^\bfB \\
\hat\bfDelta^\bfTheta\end{bmatrix}\in \R^{(p+n) \times p}$ and $\hat\bxi = \bfX^{(n)}\hat\bfB - \hat\bfTheta$.
Since $\bfM = [\bfX^{(n)};-\bfI_n]$, the estimator $(\hat\bfB,\hat\bfTheta)$  is defined as the minimizer of the cost
function
\begin{align*}
 F(\bfB, \bfTheta) = {\bigg\|\bigg(\bfM \begin{bmatrix}\bfB \\ \bfTheta\end{bmatrix}\bigg)^\top\bigg\|}_{2,1}
+ \lambda \big({\|\bfTheta\|}_{2,1} + \gamma {\|\bfB\|}_{1,1}\big).
\end{align*}
Recall that $\calJ$ and $O$ are such that $\bfB^*_{\calJ^c} = 0$ and $\bfTheta^*_{O^c,\bullet} = 0$.
This sets are interpreted  as the supports of $\bfB^*$ and $\bfTheta^*$.
The set $\calJ$ corresponds to the sparsity pattern and $O$ to the outliers.
Throughout this section, we adopt the convention that $0/0=0$.

\begin{prop}\label{prop:1}
 If, for some constant $c>1$, the penalty levels $\lambda$ and $\gamma$ satisfy the  conditions
\begin{align}
 \lambda\gamma \ge \frac{c+1}{c-1} \max_{j\in [p]} \frac{\|\bfX^{(n)}_{I,j^c}{}^\top \bepsilon_{I,j}\|_\infty}
{{\|\bepsilon_{I,j}\|}_2}\quad\text{and}\quad
\lambda \ge \frac{c+1}{c-1} \max_{i\in[n]} \bigg(\sum_{j\in [p]}\frac{\epsilon_{ij}^2}{{\|\bepsilon_{I,j}\|}_2^2}\bigg)^{1/2},\label{eq:d}
\end{align}
then the matrix $\hat\bfDelta$ belongs to the cone $\mathscr C_{\calJ,O}(c,\gamma)$.
\end{prop}

\begin{proof}
Let us define $\hat\bxi$ as the $n\times p$ matrix of estimated residuals: $\hat\bxi=\bfX^{(n)}\hat\bfB - \hat\bfTheta$.
 By definition of $\hat\bfB$ and $\hat\bfTheta$, we obtain the inequality
\begin{align*}
 {\|(\bfX^{(n)}\hat\bfB - \hat\bfTheta)^\top\|}_{2,1} + \lambda\big(\gamma {\|\hat\bfB\|}_{1,1} + {\|\hat\bfTheta\|}_{2,1}\big) &\le {\|(\bfX^{(n)}\bfB^* - \bar\bfTheta^*)^\top\|}_{2,1} + \lambda\big(\gamma {\|\bfB^*\|}_{1,1} + {\|\bar\bfTheta^*\|}_{2,1} \big) ,
\end{align*}
that can be equivalently written as
\begin{align*}
{\|\hat\bxi{}^\top\|}_{2,1} + \lambda\gamma {\|\hat\bfB\|}_{1,1} + \lambda {\|\hat\bfTheta\|}_{2,1}
\le {\|\bar\bxi^\top\|}_{2,1} + \lambda\gamma {\|\bfB^*\|}_{1,1} + \lambda {\|\bar\bfTheta^*\|}_{2,1},
\end{align*}
or as
\begin{align}
\lambda\gamma ({\|\hat\bfB\|}_{1,1}-{\|\bfB^*\|}_{1,1}) + \lambda({\|\hat\bfTheta\|}_{2,1} - {\|\bar\bfTheta^*\|}_{2,1})
\le \sum_{j\in[p]}({\|\bar\bxi_{\bullet,j}\|}_{2}  -{\|\hat\bxi_{\bullet,j}\|}_{2}) .
\label{eq:a}
\end{align}
In view of the inequality $\|a\|_2-\|b\|_2\le (a-b)^\top a/\|a\|_2$, which holds for every pair of
vectors $(a,b)$ and is a simple consequence of the Cauchy-Schwarz inequality, we have
\begin{align*}
 {\|\bar\bxi_{\bullet,j}\|}_2 -{\|\hat\bxi_{\bullet,j}\|}_2
     &\le (\bxi_{I,j}-\hat\bxi_{I,j})^\top\frac{\bxi_{I,j}}{\|\bxi_{I,j}\|_2}\\
     &= (\bxi_{I,j}-\hat\bxi_{I,j})^\top\frac{\bepsilon_{I,j}}{\|\bepsilon_{I,j}\|_2}\\
     &= (-\bfX_{I,\bullet}^{(n)}\hat\bfDelta^\bfB_{\bullet,j}+\hat\bfDelta^{\bfTheta}_{I,j})^\top\frac{\bepsilon_{I,j}}{\|\bepsilon_{I,j}\|_2}.
\end{align*}
Summing these inequalities over all $j\in [p]$ and applying the duality inequalities we infer that
\begin{align*}
\sum_{j\in[p]}({\|\bxi_{\bullet,j}\|}_{2}  -{\|\hat\bxi_{\bullet,j}\|}_{2})
     &\le-\sum_{j\in [p]}(\bfX_{I,\bullet}^{(n)}\hat\bfDelta^{\bfB}_{\bullet,j})^\top\frac{\bepsilon_{I,j}}{\|\bepsilon_{I,j}\|_2}
		   +\sum_{i\in I}\sum_{j\in[p]}\hat\bfDelta^{\bfTheta}_{i,j}\frac{\bepsilon_{i,j}}{\|\bepsilon_{I,j}\|_2}\\
		 &\le \sum_{j\in[p]} \|\hat\bfDelta^{\bfB}_{\bullet,j}\|_1
		   \frac{\|\bfX_{I,j^c}^{(n)}{}^\top\bepsilon_{I,j}\|_\infty}{\|\bepsilon_{I,j}\|_2}
			 + \sum_{i\in[n]}\|\hat\bfDelta^{\bfTheta}_{i,\bullet}\|_2 \bigg(\sum_{j\in[p]}
			 \frac{\epsilon_{ij}^2}{\|\bepsilon_{I,j}\|_2^2}\bigg)^{1/2}.
\end{align*}
When condition (\ref{eq:d}) is satisfied, the last inequality yields
\begin{align*}
\sum_{j\in[p]}({\|\bar\bxi_{\bullet,j}\|}_{2}  -{\|\hat\bxi_{\bullet,j}\|}_{2})
		 &\le \bigg(\frac{c-1}{c+1}\bigg)\bigg(\lambda\gamma\sum_{j\in[p]} \|\hat\bfDelta^{\bfB}_{\bullet,j}\|_1
			 + \lambda\sum_{i\in[n]}\|\hat\bfDelta^{\bfTheta}_{i,\bullet}\|_2 \bigg)\\
		 &=  \lambda\bigg(\frac{c-1}{c+1}\bigg)\big(\gamma\|\hat\bfDelta^{\bfB}\|_{1,1}
			 + \|\hat\bfDelta^{\bfTheta}\|_{2,1} \big).
\end{align*}
This inequality, in conjunction with Eq.\ (\ref{eq:a}), implies that
\begin{align}
\gamma ({\|\hat\bfB\|}_{1,1}-{\|\bfB^*\|}_{1,1}) + ({\|\hat\bfTheta\|}_{2,1} - {\|\bar\bfTheta^*\|}_{2,1})
\le \bigg(\frac{c-1}{c+1}\bigg)\big(\gamma\|\hat\bfDelta^{\bfB}\|_{1,1}
			 + \|\hat\bfDelta^{\bfTheta}\|_{2,1} \big).
\label{eq:b}
\end{align}
On the other hand, using the triangle inequality and the fact that $\bfB_{\calJ^c}^*=\bfTheta_{O^c,\bullet}^*=0$, we get
\begin{align*}
{\|\hat\bfB\|}_{1,1}-{\|\bfB^*\|}_{1,1}
   &= {\|\hat\bfB_{\calJ^c}\|}_{1,1}+{\|\hat\bfB_{\calJ}\|}_{1,1}-{\|\bfB^*_{\calJ}\|}_{1,1}\\
	 &\ge  {\|\hat\bfDelta^{\bfB}_{\calJ^c}\|}_{1,1}-{\|\hat\bfDelta^\bfB_{\calJ}\|}_{1,1},\\
{\|\hat\bfTheta\|}_{2,1} - {\|\bar\bfTheta^*\|}_{2,1}
   &= {\|\hat\bfTheta_{O^c,\bullet}\|}_{2,1}+{\|\hat\bfTheta_{O,\bullet}\|}_{2,1}-{\|\bar\bfTheta^*_{O,\bullet}\|}_{2,1}\\
   &\ge {\|\hat\bfDelta^\bfTheta_{O^c,\bullet}\|}_{2,1}-{\|\hat\bfDelta^\bfTheta_{O,\bullet}\|}_{2,1}.
\end{align*}
The combination of these bounds with Eq.\ \eqref{eq:b} leads to
\begin{align*}
 {\|\hat\bfDelta^\bfB_{\calJ^c}\|}_{1,1} + \gamma^{-1} {\|\hat\bfDelta^\bfTheta_{O^c,\bullet}\|}_{2,1} &\le c \big({\|\hat\bfDelta^\bfB_{\calJ}\|}_{1,1} + \gamma^{-1} {\|\hat\bfDelta^\bfTheta_{O,\bullet}\|}_{2,1}\big),
\end{align*}
which completes the proof of the proposition.
\end{proof}

The following lemmas prepare the proof of Theorem \ref{thm:3}.
Lemma \ref{lem:a} presents an inequality obtained by writing the KKT conditions for the cost function $F$.

\begin{lemma}\label{lem:a} There exists a $n\times p$ matrix $\bfV$ is such that
\begin{align}
\|\bfV_{i,\bullet}\|_2\le 1,\qquad
\bfV_{i,\bullet}^\top\hat\bfTheta_{i,\bullet} = \|\hat\bfTheta_{i,\bullet}\|_2,\qquad\forall\,i\in[n]
\end{align}
and, for every $j\in[p]$ such that $\hat\bxi_{\bullet,j}\neq 0$, the following inequality holds
\begin{align}\label{eq:1}
{\|\bfM\hat\bfDelta_{\bullet,j}\|}_2^2
    &\le  -\bar\bxi_{\bullet,j}^\top\bfM\hat\bfDelta_{\bullet,j}
					- \lambda {\|\hat\bxi_{\bullet,j}\|}_2 \bfV_{\bullet,j}^\top\hat\bfDelta^\bfTheta_{\bullet,j}
			+ \lambda\gamma {\|\hat\xi_{\bullet,j}\|}_2\big({\|\bfB^*_{j^c,j}\|}_1 - {\|\hat\bfB_{j^c,j}\|}_1 \big).
\end{align}

\end{lemma}
\begin{proof}
Recall that the estimator $(\hat\bfB,\hat\bfTheta)$ minimizes the cost function
\begin{align}\label{eq:cost:2}
F(\bfB, \bfTheta) = \sum_{j=1}^p {\|\bfX^{(n)}_{\bullet,j} + \bfX^{(n)}_{\bullet,j^c}\bfB_{j^c,j}-\bfTheta_{\bullet,j}\|}_2
     +\lambda\gamma \sum_{j=1}^p{\|\bfB_{\bullet,j}\|}_1+\lambda \sum_{i=1}^n{\|\bfTheta_{i,\bullet}\|}_2.
\end{align}
According to the KKT conditions, this convex function is minimized at $(\hat\bfB,\hat\bfTheta)$ if and only if
the zero vector belongs to the sub-differential of $F$ at $(\hat\bfB,\hat\bfTheta)$, denoted by $\partial F(\hat\bfB,\hat\bfTheta)$.
This entails, in particular, that for every $j\in[p]$, $\mathbf 0_{p-1+n}\in \partial_{(\bfB_{j^c,j},\bfTheta_{\bullet,j})}
F(\hat\bfB,\hat\bfTheta)$. In other terms, there exist vectors $\bfu_j\in \partial_{(\bfB_{j^c,j},\bfTheta_{\bullet,j})}
{\|\bfX^{(n)}_{\bullet,j} + \bfX^{(n)}_{\bullet,j^c}\hat\bfB_{j^c,j}-\hat\bfTheta_{\bullet,j}\|}_2$, $\bfw_j\in
\partial_{(\bfB_{j^c,j},\bfTheta_{\bullet,j})} {\|\hat\bfB_{\bullet,j}\|}_1$
and $\bfv_j\in \partial_{(\bfB_{j^c,j},\bfTheta_{\bullet,j})}\sum_{i=1}^n{\|\hat\bfTheta_{i,\bullet}\|}_2$ such that
$\bfu_j+\lambda\gamma\bfw_j+\lambda\bfv_j=0$. Since we assume that ${\|\hat\bxi_{\bullet,j}\|}_2>0$, the first partial sub-differential
out of three appearing in the previous sentence is actually a differential and thus $\bfu_j = [\bfX^{(n)}_{\bullet,j^c} ; -\bfI_n]^\top
(\bfX^{(n)}\hat\bfB_{\bullet,j}-\hat\bfTheta_{\bullet,j})/{\|\hat\bxi_{\bullet,j}\|}_2$. After a multiplication by ${\|\hat\bxi_{\bullet,j}\|}_2$, we get
\begin{align*}
[\bfX^{(n)}_{\bullet,j^c} ; -\bfI_n]^\top(\bfX^{(n)}\hat\bfB_{\bullet,j}-\hat\bfTheta_{\bullet,j}) = -\lambda\gamma {\|\hat\bxi_{\bullet,j}\|}_2\bfw_j - \lambda
{\|\hat\bxi_{\bullet,j}\|}_2\bfv_j.
\end{align*}
This equation (combined with relation \eqref{eq:xi}) can be equivalently written as
\begin{align*}
[\bfX^{(n)}_{\bullet,j^c} ; -\bfI_n]^\top \bfM\hat\bfDelta_{\bullet,j} = -[\bfX^{(n)}_{\bullet,j^c} ; -\bfI_n]^\top\bar\bxi_{\bullet,j} -\lambda\gamma {\|\hat\bxi_{\bullet,j}\|}_2\bfw_j - \lambda
{\|\hat\bxi_{\bullet,j}\|}_2\bfv_j.
\end{align*}
We take the scalar product of the both sides of this relation with the vector $\hat\bfDelta_{j^c,j}$ and, using the fact that $\hat\bfDelta_{j,j}=0$, we obtain
\begin{align*}
\|\bfM\hat\bfDelta_{\bullet,j}\|_2^2
        = -\hat\bfDelta_{\bullet,j}^\top\bfM^\top\bar\bxi_{\bullet,j}
          -\lambda\gamma {\|\hat\bxi_{\bullet,j}\|}_2\hat\bfDelta_{\bullet,j}^\top\bfw_j
          - \lambda{\|\hat\bxi_{\bullet,j}\|}_2\hat\bfDelta_{\bullet,j}^\top\bfv_j.
\end{align*}
The desired inequality follows by setting $\bfV = [(\bfv_1)_{p:(p-1+n)},\ldots,(\bfv_p)_{p:(p-1+n)}]$ and by using the following simple properties of the sub-differentials of the
$\ell_1$ and $\ell_2$-norms:
\begin{align*}
(\bfw_j)_l &= 0,\qquad \forall l\ge p,\\
|(\bfw_j)_l| &\le  1,\qquad \forall l\in[p-1],\\
(\bfw_j)_{1:(p-1)}^\top &\hat\bfB_{j^c,j} = \|\hat\bfB_{j^c,j}\|_1,\\
(\bfv_j)_l &= 0,\qquad \forall l\in[p-1],\\
(\bfv_j)_{p-1+i} &=  \frac{\hat\bfTheta_{i,j}}{\|\hat\bfTheta_{i,\bullet}\|_2},\qquad
        \begin{cases}
            i\in[n],\\
            \|\hat\bfTheta_{i,\bullet}\|_2>0,
        \end{cases}
        \\
|(\bfv_j)_{p-1+i}| & \le \frac{|\btheta_j|}{\|\btheta\|_2},
        \qquad\qquad
        \begin{cases}
            i\in[n],\\
            \|\hat\bfTheta_{i,\bullet}\|_2=0,\\
            \forall \btheta \in \R^p, \|\btheta\|_2 > 0.
        \end{cases}
\end{align*}
Indeed, the first three relations imply that $-\hat\bfDelta_{\bullet,j}^\top\bfw_j\le \|\bfB^*_{\bullet,j}\|_1-\|\hat\bfB_{\bullet,j}\|_1$ while the three
last relations yield $\hat\bfDelta_{\bullet,j}^\top\bfv_j = \bfV_{\bullet,j}^\top\hat\bfDelta^\bfTheta_{\bullet,j}$ along with
$\|\bfV_{i,\bullet}\|_2\le 1$ and $\bfV_{i,\bullet} \hat\bfTheta_{i,\bullet}^\top=\|\hat\bfTheta_{i,\bullet}\|_2$.
\end{proof}

\begin{lemma}\label{lem:b}
If inequality \eqref{eq:1}  is true, then
\begin{align}\label{eq:2}
{\|\bfM\hat\bfDelta_{\bullet,j}\|}_2^2
    &\le  -2\bar\bxi_{\bullet,j}^\top\bfM\hat\bfDelta_{\bullet,j}
					- 2\lambda {\|\bar\bxi_{\bullet,j}\|}_2 \bfV_{\bullet,j}^\top\hat\bfDelta^\bfTheta_{\bullet,j}
			+ 2\lambda\gamma {\|\bar\bxi_{\bullet,j}\|}_2\big({\|\bfB^*_{j^c,j}\|}_1 - {\|\hat\bfB_{j^c,j}\|}_1 \big)\notag\\
    &\qquad+\lambda^2(\gamma{\|\hat\bfDelta_{\bullet,j}^\bfB\|}_1+|\bfV_{\bullet,j}^\top\hat\bfDelta^\bfTheta_{\bullet,j}|)^2.
\end{align}
\end{lemma}
\begin{proof}
This is a direct consequence of Lemma~\ref{lem:aa} (with $R= {\|\bfM\hat\bfDelta_{\bullet,j}\|}_2$) and the fact
that $|{\|\hat\xi_{\bullet,j}\|}_2-{\|\bar\xi_{\bullet,j}\|}|_2\le {\|\bfM\hat\bfDelta_{\bullet,j}\|}_2$.
\end{proof}

\begin{lemma}\label{lem:bb}
If inequality \eqref{eq:2}  is true and if the penalty levels $\lambda$ and $\gamma$ satisfy conditions \eqref{eq:d} for some constant $c>1$, then
\begin{align}\label{eq:3}
{\|\bfM\hat\bfDelta\|}_F^2
    &\le 4\lambda c {\|\bxi_{I,\bullet}^\top\|}_{2,\infty} \big( \gamma {\|\hat\bfDelta_\calJ^\bfB\|}_{1,1} + {\|\hat\bfDelta_{O,\bullet}^\bfTheta\|}_{2,1}\big) + \lambda^2(1+c)^2 \big( \gamma {\|\hat\bfDelta_\calJ^\bfB\|}_{1,1} + {\|\hat\bfDelta_{O,\bullet}^\bfTheta\|}_{2,1}\big)^2 .
\end{align}
\end{lemma}
\begin{proof}
We begin by noting that for a $n\times p$ matrix $\bfV$ that satisfies ${\|\bfV_{i,\bullet}\|}_2 \le 1$ for any $i$ belonging to $[n]$, the Cauchy-Schwarz inequality yields that
 \begin{align}\label{eq:lem:bb:1}
  \sum_{j=1}^p |\bfV_{\bullet,j}^\top\hat\bfDelta^\bfTheta_{\bullet,j}| \le \sum_{i=1}^n \sum_{j=1}^p |\bfV_{i,j}\hat\bfDelta^\bfTheta_{i,j}| \le \sum_{i=1}^n {\|\bfV_{i,\bullet}\|}_2 {\|\hat\bfDelta^\bfTheta_{i,\bullet}\|}_2 \le {\|\hat\bfDelta^\bfTheta\|}_{2,1} .
 \end{align}
We also deduce
 \begin{align}
  \sum_{j=1}^p \big( \gamma {\|\hat\bfDelta_{\bullet,j}^\bfB\|}_1 + |\bfV_{\bullet,j}^\top \hat\bfDelta^\bfTheta_{\bullet,j}| \big)^2 &\le \Big(\sum_{j=1}^p \gamma {\|\hat\bfDelta_{\bullet,j}^\bfB\|}_1 + |\bfV_{\bullet,j}^\top \hat\bfDelta^\bfTheta_{\bullet,j}| \Big)^2 \notag \\
  &\le \big( \gamma {\|\hat\bfDelta^\bfB\|}_{1,1} + {\|\hat\bfDelta^\bfTheta\|}_{2,1} \big)^2 . \label{eq:lem:bb:2}
 \end{align}
Besides, it holds
 \begin{align*}
  - \sum_{j = 1}^p \bar\bxi_{\bullet,j}^\top\bfM\hat\bfDelta_{\bullet,j} &= \sum_{j = 1}^p \big(\hat\bfDelta^\bfTheta_{\bullet,j} - \bfX^{(n)} \hat\bfDelta^\bfB_{\bullet,j} \big)^\top \bar\bxi_{\bullet,j} \\
  &= \sum_{j = 1}^p {\|\bar\bxi_{\bullet,j}\|}_2 \Big(\sum_{i\in I} \hat\bfDelta^\bfTheta_{i,j} \frac{\bepsilon_{i,j}}{{\|\bepsilon_{\bullet,j}\|}_2} - \hat\bfDelta^\bfB_{\bullet,j}{}^\top \bfX^{(n)}_{I,\bullet}{}^\top \frac{\bepsilon_{I,j}}{{\|\bepsilon_{I,j}\|}_2} \Big) \\
  &\le (\max_{j\in[p]} {\|\bxi_{I,j}\|}_2) \Big(\sum_{i\in I} \sum_{j = 1}^p \frac{|\hat\bfDelta^\bfTheta_{i,j}\bepsilon_{i,j}|}{{\|\bepsilon_{I,j}\|}_2} + \sum_{j = 1}^p \frac{|\hat\bfDelta^\bfB_{\bullet,j}{}^\top \bfX^{(n)}_{I,\bullet}{}^\top \bepsilon_{I,j}|}{{\|\bepsilon_{I,j}\|}_2} \Big) ,
 \end{align*}
thus, by the duality inequality $|\hat\bfDelta^\bfB_{\bullet,j}{}^\top \bfX^{(n)}_{I,\bullet}{}^\top \bepsilon_{I,j}|
\le {\|\hat\bfDelta^\bfB_{\bullet,j}\|}_1 {\|\bfX^{(n)}_{I,\bullet}{}^\top \bepsilon_{I,j}\|}_\infty$ and the Cauchy-Schwarz inequality, and as the penalty levels satisfy conditions \eqref{eq:d}, we find
 \begin{align}
 - \sum_{j = 1}^p \bar\bxi_{\bullet,j}^\top\bfM\hat\bfDelta_{\bullet,j} &\le {\|\bxi_{I,\bullet}^\top\|}_{2,\infty} \Big(\sum_{i\in I} {\|\hat\bfDelta^\bfTheta_{i,\bullet}\|}_2 \Big(\sum_{j = 1}^p \frac{\bepsilon_{i,j}^2}{{\|\bepsilon_{I,j}\|}_2^2}\Big)^\frac12  + \sum_{j = 1}^p {\|\hat\bfDelta^\bfB_{\bullet,j}\|}_1 \frac{{\|\bfX^{(n)}_{I,\bullet}{}^\top \bepsilon_{I,j}\|}_\infty}{{\|\bepsilon_{I,j}\|}_2} \Big) \notag \\
 &\le \lambda\frac{c-1}{c+1} \big(\gamma {\|\hat\bfDelta^\bfB\|}_{1,1} + {\|\hat\bfDelta^\bfTheta\|}_{2,1} \big)
{\|\bxi_{I,\bullet}^\top\|}_{2,\infty}. \label{eq:lem:bb:3}
\end{align}
From inequality \eqref{eq:2}, we get
 \begin{align*}
{\|\bfM\hat\bfDelta_{\bullet,j}\|}_2^2
    &\le  -2\bar\bxi_{\bullet,j}^\top\bfM\hat\bfDelta_{\bullet,j}
					- 2\lambda {\|\bar\bxi_{\bullet,j}\|}_2 \Big( \bfV_{\bullet,j}^\top\hat\bfDelta^\bfTheta_{\bullet,j}
			+ \gamma \big( {\|\hat\bfB_{j^c,j}\|}_1 - {\|\bfB^*_{j^c,j}\|}_1 \big)\Big)\notag\\
    &\qquad+\lambda^2\big(\gamma{\|\hat\bfDelta_{\bullet,j}^\bfB\|}_1+|\bfV_{\bullet,j}^\top\hat\bfDelta^\bfTheta_{\bullet,j}|\big)^2 ,
\end{align*}
for every $j \in [p]$.
Then, summing over all $j$ and using the triangle inequality, we have
 \begin{align*}
{\|\bfM\hat\bfDelta\|}_F^2
    &\le  -2 \sum_{j = 1}^p \bar\bxi_{\bullet,j}^\top\bfM\hat\bfDelta_{\bullet,j}
					+ 2 \lambda {\|\bxi_{I,j}\|}_2 \Big( |\bfV_{\bullet,j}^\top\hat\bfDelta^\bfTheta_{\bullet,j}|
			+ \gamma {\|\hat\bfDelta_{\bullet,j}^\bfB\|}_{1}\Big)\notag\\
    &\qquad+\lambda^2\big(\gamma{\|\hat\bfDelta_{\bullet,j}^\bfB\|}_1+|\bfV_{\bullet,j}^\top\hat\bfDelta^\bfTheta_{\bullet,j}|\big)^2 .
\end{align*}
Combining the latter with equations \eqref{eq:lem:bb:1}, \eqref{eq:lem:bb:2} and \eqref{eq:lem:bb:3}, we arrive at
 \begin{align*}
{\|\bfM\hat\bfDelta\|}_F^2
    &\le \lambda \frac{4c}{c+1} \big(\gamma {\|\hat\bfDelta^\bfB\|}_{1,1} + {\|\hat\bfDelta^\bfTheta\|}_{2,1} \big)
		 {\|\bxi_{I,\bullet}^\top\|}_{2,\infty} + \lambda^2\big(\gamma {\|\hat\bfDelta^\bfB\|}_{1,1} + {\|\hat\bfDelta^\bfTheta\|}_{2,1} \big)^2 ,
\end{align*}
we finally apply Proposition \ref{prop:1} that gives inequality \eqref{eq:3}.
\end{proof}

\begin{prop}\label{propo:1}
Choose $\gamma=1$ and $\delta\in(0,1)$ such that $n\ge |O|+16\log(2p/\delta)$ and choose
\begin{equation}\label{lambd2}
\lambda = 6 \bigg(\frac{\log(2np/\delta)}{n-|O|}\bigg)^{1/2}.
\end{equation}
Then
\begin{enumerate}
\item[i)] with probability at least $1-\delta$, the penalty levels $\lambda$ and $\gamma$ satisfy conditions \eqref{eq:d}
for some constant $c=2$.
\item[ii)]	If\ $4\lambda(|\calJ|^{1/2}+ |O|^{1/2}) < \kappa^{1/2}$ holds, then there exists an event $\mathcal E_0$
of probability at least $1-2\delta$ such that in\footnote{Recall that $\mathcal E_\kappa$ is the event defined
by  \eqref{eq:mcc}} $\mathcal E_\kappa\cap \mathcal E_0$, we have
 \begin{align}
  {\|\bfM\hat\bfDelta\|}_{2,2} &\le \frac{C_2}{\sqrt{\kappa}}\max_{j\in[p]} (\omega^*_{jj})^{-1/2} \big(|\calJ|^{1/2} + |O|^{1/2}\big)\bigg(\frac{\log(2np/\delta)}{n-|O|}\bigg)^{1/2} , \label{eq:4} \\
{\|\hat\bfDelta^\bfB\|}_{1,1}&+{\|\hat\bfDelta^\bfTheta\|}_{2,1}
\le \frac{12C_2}{\kappa}\max_{j\in[p]} (\omega^*_{jj})^{-1/2} \big(|\calJ| + |O|\big)\bigg(\frac{\log(2np/\delta)}{n-|O|}\bigg)^{1/2}
\label{eq:5}
 \end{align}
 with $C_2 \le 75$.
\end{enumerate}
\end{prop}

\begin{proof}
Claim i) of the proposition is obtained by standard arguments relying on tail bounds for Gaussian and $\chi^2$ distributions and
the union bound. These arguments are similar to those presented in Section~\ref{subs:prob} and, therefore, are skipped.

Analogously,  using Lemma~\ref{lem:3}, we find that with probability at least $1-\delta$, we have
$\|\bxi_{I,\bullet}\|_{2,\infty}\le (1+2^{-3/2})\max_j (\omega_{jj}^*)^{-1/2}$. We denote by $\mathcal E_0$ the intersection of this event
with the one of claim i). By the union bound, we have $\bfP(\mathcal E_0)\ge 1-2\delta$. In the rest of this proof, we place ourselves in the
vent $\mathcal E_0\cap \mathcal E_\kappa$. By the compatibility assumption (event $\mathcal E_\kappa$), we have
 \begin{align}\label{comp}
  {\|\hat\bfDelta_\calJ^\bfB\|}_{1,1} \le \frac{|\calJ|^{1/2}}{\kappa^{1/2}} {\|\bfM\hat\bfDelta\|}_{2,2} \quad \text{and} \quad {\|\hat\bfDelta_{O,\bullet}^\bfTheta\|}_{2,1} \le \frac{|O|^{1/2}}{\kappa^{1/2}} {\|\bfM\hat\bfDelta\|}_{2,2}.
 \end{align}
Since in the event $\mathcal E_0$ the conditions of Lemma \ref{lem:bb} are met, inequality \eqref{eq:3} readily implies
inequality \eqref{eq:4}. On the other hand, we know from Proposition~\ref{prop:1} that $\hat\bfDelta$ belongs to the dimension
reduction cone $\mathscr C_{\calJ,O}(2,1)$. Therefore,
 \begin{align*}
  {\|\hat\bfDelta^\bfB\|}_{1,1}+{\|\hat\bfDelta^\bfTheta\|}_{2,1} \le 3({\|\hat\bfDelta_\calJ^\bfB\|}_{1,1}+{\|\hat\bfDelta_{O,\bullet}^\bfTheta\|}_{2,1}) \le \frac{3(|\calJ|\vee |O|)^{1/2}}{\kappa^{1/2}} {\|\bfM\hat\bfDelta\|}_{2,2}.
 \end{align*}
Using the upper bound on ${\|\bfM\hat\bfDelta\|}_{2,2}$ provided by \eqref{eq:4}, we immediately obtain bound \eqref{eq:5}.
\end{proof}

\end{document}